\documentclass[aos,preprint]{imsart}
\RequirePackage{amsthm,amsmath}
\RequirePackage[numbers]{natbib}
\RequirePackage[colorlinks,citecolor=blue,urlcolor=blue]{hyperref}

\usepackage{graphicx,bbm,amssymb,amsmath,amsfonts,amsthm}% ,fullpage,}

\usepackage[english]{babel}
\usepackage{amsmath, amssymb, amsfonts,amsthm}
\usepackage{stmaryrd}

\usepackage{epsfig}
\psfigdriver{dvips}

\usepackage{hyperref}
\usepackage{hypernat}
\usepackage{appendix}
\usepackage{xspace}
\theoremstyle{plain}
\newtheorem{theorem}{Theorem}
\newtheorem{lemma}[theorem]{Lemma}

\newtheorem{corollary}[theorem]{Corollary}
\newtheorem{proposition}[theorem]{Proposition}
\theoremstyle{definition}

\theoremstyle{remark}
\newtheorem{remark}{Remark}

\newcommand{\egaldef}{:=} % egalite definissant la quantite de gauche
\newcommand{\flens}{\rightarrow} % fleche d'application X->Y (ensembles)
\newcommand{\telque}{\, \mbox{ s.t. } \,} % tel que dans une definition d'ensemble

\newcommand{\R}{\mathbb{R}} %corps des reels
\newcommand{\Z}{\mathbb{Z}} %anneau des entiers
\newcommand{\N}{\mathbb{N}} %entiers naturels
\newcommand{\Xcal}{\mathcal{X}}
\newcommand{\Xbf}{\mathbf{X}}
\newcommand{\Ycal}{\mathcal{Y}}
\newcommand{\Pcal}{\mathcal{P}}
\newcommand{\B}{\mathcal{B}}
\newcommand{\Dcal}{\mathcal{D}}
\newcommand{\Lcal}{\mathcal{L}}
\newcommand{\G}{\mathcal{G}}
\newcommand{\At}{\tilde{A}}
\newcommand{\Bt}{\tilde{B}}
\newcommand{\Ct}{\tilde{C}}
\newcommand{\Boule}{\mathbb{B}}

\DeclareMathOperator{\card}{Card} % cardinal
\newcommand{\paren}[1]{\left( \left. #1 \right. \right)} 
\newcommand{\croch}[1]{\left[ \left. #1 \right. \right]} 
\newcommand{\set}[1]{\left\{ \left. #1 \right. \right\}}
\newcommand{\absj}[1]{\left\lvert #1 \right\rvert} %joli abs
\providecommand{\norm}[1]{\left \lVert #1 \right\rVert}

\newcommand{\PESup}[1]{\left\lceil#1\right\rceil} %Partie Entiere superieure
\newcommand{\psh}[2]{\ensuremath{\langle #1,#2\rangle}\xspace}

\renewcommand{\P}{\mathbb{P}}
\newcommand{\E}{\mathbb{E}} %esperance
\DeclareMathOperator{\var}{var} %variance
\DeclareMathOperator{\med}{Med} %median
\newcommand{\sachant}{\, \right| \left. \,} % pour l'esp\'erance conditionnelle...
\newcommand{\ERM}{\widehat{s}}%ERM
\newcommand{\ERMP}{\breve{s}}

\newcommand{\Ical}{\mathcal{I}}
\newcommand{\M}{\mathcal{M}}
\newcommand{\mM}{m \in \M}

\newcommand{\mh}{\widehat{m}}

\DeclareMathOperator{\pen}{pen}%penalite
\DeclareMathOperator{\crit}{crit}%critere empirique
\newcommand{\critid}{\crit_{\mathrm{id}}} % penalite ideale
\newcommand{\hyptag}[1]{\tag{\ensuremath{\mathbf{#1}}}} % nom d'hypothese

\newcommand{\Il}{I_{\lambda}}
\newcommand{\lL}{\lambda\in\Lambda}

\newcommand{\llpL}{\lambda,\lambda^{\prime} \in \Lambda}
\newcommand{\lp}{\lambda^{\prime}}
\newcommand{\Ks}{K_{\star}}
\newcommand{\st}{s_{\star}}

\newcommand{\psil}{\psi_{\lambda}} % vecteurs de base de S_m
\newcommand{\psilp}{\psi_{\lambda^{\prime}}}

\newcommand{\betal}{\beta_{\lambda}} % coefficient
\newcommand{\al}{a_{\lambda}}
\newcommand{\thetal}{\theta_{\lambda}}
\newcommand{\alp}{a_{\lambda^{\prime}}}

\newcommand{\betahl}{\widehat{\beta}_{\lambda}} % coefficient
\newcommand{\ahl}{\widehat{a}_{\lambda}}
\newcommand{\thetahl}{\widehat{\theta}_{\lambda}}

\newcommand{\sthet}{s_{\theta}}

\newcommand{\omegal}{\omega_{\lambda}}

\newcommand{\thetah}{\widehat{\theta}}

\newcommand{\betlassol}{\widehat{a}_{\lambda}}

\newcommand{\un}{\ensuremath{\mathbf{1}}}

\newcommand{\olP}{\ensuremath{\overline{P}}}

\newtheorem{postita}{Post-it}

\startlocaldefs
\endlocaldefs

\begin{document}

\begin{frontmatter}

% "Title of the paper"
\title{Robust empirical mean Estimators}
\runtitle{Robust empirical mean Estimators}

% indicate corresponding author with \corref{}
% \author{\fnms{John} \snm{Smith}\corref{}\ead[label=e1]{smith@foo.com}\thanksref{t1}}
% \thankstext{t1}{Thanks to somebody} 
% \address{line 1\\ line 2\\ printead{e1}}
% \affiliation{Some University}

\begin{aug}
\author{\fnms{Matthieu} \snm{Lerasle}\ead[label=e1]{mlerasle@unice.fr}} \and
\author{\fnms{Roberto I.} \snm{Oliveira}\ead[label=e1]{rimfo@impa.br}}

%\thankstext{t1}{supported by FAPESP grant 2009/09494-0. This work is part of USP project ``Mathematics, computation, language and the brain''.}
%\thankstext{t2}{supported by COFECUB USP 2010 project \emph{Syst\`emes stochastiques \`a interaction de port\'ee variable}}
\runauthor{M. Lerasle and R. Oliveira}

\affiliation{CNRS-Universit\'e Nice and IMPA}

\address{
Estrada Da. Castorina, 110. \\
Rio de Janeiro, RJ, Brazil. 22460-320\printead{e1}\\
}
\end{aug}

\begin{abstract}: 
We study robust estimators of the mean of a probability measure $P$, called robust empirical mean estimators. This elementary construction is then used to revisit a problem of aggregation and a problem of estimator selection, extending these methods to not necessarily bounded collections of previous estimators. 

We consider then the problem of robust $M$-estimation. We propose a slightly more complicated construction to handle this problem and, as examples of applications, we apply our general approach to least-squares density estimation, to density estimation with K\"ullback loss and to a non-Gaussian, unbounded, random design and heteroscedastic regression problem. 

Finally, we show that our strategy can be used when the data are only assumed to be mixing. 
\end{abstract}

\end{frontmatter}

\noindent
{\small {{\bf Keywords and phrases :} Robust estimation, aggregation, estimator selection, density estimation, heteroscedastic regression.} }

%\tableofcontents

\section{Introduction}

The goal of this paper is to develop the theory and applications of what we call robust empirical mean estimators. As a first step, we show that one can estimate the mean $m\egaldef \E_{P}(X)$ of a random variable $X$ with distribution $P$ based on the observation of a finite sample $X_{1:n}\egaldef (X_{1},\ldots,X_{n})$ i.i.d with marginal $P$. More precisely, our estimator $\mh(\delta)$ takes as input the sample $X_{1:n}$ and a confidence level $\delta\in (0,1)$, and satisfies:
\begin{equation}\label{eq.obj}
\P\set{\mh(\delta)-m>C\sigma\sqrt{\frac{\ln(\delta^{-1})}n}}\leq \delta\enspace
\end{equation}
where $\sigma^2$ is the variance of $X$ and $C$ is a universal constant. The classical empirical mean estimator $\mh\egaldef n^{-1}\sum_{i=1}^{n}X_{i}$ satisfies such bounds only when the observations are Gaussian. Otherwise, some extra assumptions on the data are generally required and the concentration of $\mh$ takes a different shape. For example, if $X_{1}\leq b$, Bennett's concentration inequality gives 
\[
\P\set{\mh-m>\sigma\sqrt{2\frac{\ln(\delta^{-1})}n}+\frac{b\ln(\delta^{-1})}{3n}}\leq \delta\enspace, 
\]
and heavier tails imply wider confidence intervals for the empirical mean. 

The goal of achieving \eqref{eq.obj} is not new. Catoni \cite{Ca01, Ca10} has recently developed a Pac-Bayesian approach that has applied in several problems of non parametric statistics including linear regression and classification \cite{Al08,AC10, Ca07}. The main idea behind the constructions \cite{Ca10} is to do ``soft-truncation" of the data, so as to mitigate heavy tails. This gives estimators achieving (\ref{eq.obj}) with nearly optimal constant $C\sim \sqrt{2}$ for a wide range of $\delta$.

One significant drawback of Catoni's construction is that one needs to know $\sigma$ or an upper bound for the kurtosis $\kappa$ in order to achieve the bound in (\ref{eq.obj}). In this paper we show that no such information is necessary. The construction we present, while not well-known, is not new: we learned about it from \cite{HsuBlog} and the main ideas seem to date back to the work of Nemirovsky and Yudin \cite{NeYuBook}. The main idea of these authors was to split the sample into regular blocks, then take the empirical mean on each block, and finally define our estimator as the median of these preliminary estimators. When the number $V$ of blocks is about $\ln(\delta^{-1})$, the resulting robust empirical mean estimator satisfies \eqref{eq.obj}. 

This result is sufficient to revisit some recent procedures of aggregation and selection of estimators and extend their range of application to heavier-tailed distributions. Before we move on to discuss these applications, we note that in all cases the desired confidence level $\delta$ will be built into the estimator (ie. it must be chosen in advance), and it cannot be smaller than $e^{-n/2}$. While similar limitations are faced by Catoni in \cite{Ca10}, the linear regression estimator in \cite{AC10} manages to avoid this difficulty. By contrast, our estimators are essentially as efficient as their non-robust counterparts. This favourable trait is not shared by the estimator in \cite{AC10}, which is based on a non-convex optimization problem. 

We now discuss our first application of robust estimator, to the least-squares density estimation problem. We study the Lasso estimator of \cite{Ti96}, which is a famous aggregation procedure extensively studied over the last few years (see for example \cite{BRT09, BTWB, CDS01, DET06, GR04, Lo08,BM06, ZH08, ZY07, Zo06} and the references therein), using $\ell_{1}$-penalties. We modify the Lasso estimator, adapted to least-squares density estimation in \cite{BTWB}, and extend the results of \cite{BTWB} to unbounded dictionaries. 

We then study the estimator selection procedure of \cite{Ba11} in least-squares density estimation. Estimator selection is a new important theory that covers in the same framework the problems of model selection and of the selection of a statistical strategy together with the tuning constants associated. \cite{Ba11} provides a general approach including density estimation but his procedure based on pairwise tests between estimators, in the spirit of \cite{Bi06} was not computationally tractable. A tractable algorithm is available in Gaussian linear regression in \cite{BGH11}. Our contribution to the field is twofold. First, we extend the tractable approach of \cite{BGH11} to least-squares density estimation. Our results rely on empirical process theory rather than Gaussian techniques and may be adapted to other frameworks, provided that some uniform upper bounds on the previous estimators are available. Then, we extend these results using robust empirical means. The robust approach is restricted to the variable selection problem but allows to handle not necessarily bounded collections of estimators.

After these direct applications, we consider the problem of building robust estimators in $M$-estimation. We present a general approach based on a slightly more complicated version of the basic robust empirical mean construction. We introduce a margin type assumption to control the risk of our estimator. This assumption is not common in the literature but we present classical examples of statistical frameworks where it is satisfied. We apply then the general strategy to least-squares and maximum likelihood estimators of the density and to a non-Gaussian, non bounded, random design and heteroscedastic regression setting. In these applications, our estimators satisfy optimal risk bounds, up to a logarithmic factor. This is an important difference with \cite{AC10} where exact convergence rates were obtained using a localized version of the Pac-Bayesian estimators. 

Finally, the decomposition of the data set into blocks was used in recent works \cite{BCV01,CM02,Le09,Le10} in order to extend model selection procedures to mixing data. The decomposition is done to couple the data with independent blocks and then apply the methods of the independent case. Similar ideas can be developed here, but it is interesting to notice that the extension does not require another block decomposition. The initial decomposition of the robust empirical mean algorithm is sufficient and our procedures can be used with mixing data. We extend several theorems of the previous sections to illustrate this fact.

The paper is organized as follows. Section \ref{sect.RobEstMean} presents basic notations, the construction of the new estimators and the elementary concentration inequality that they satisfy. As a first application, we deduce an upper bound for the variance and a confidence interval for the mean. Section \ref{sect.FirstApp} presents the application to Lasso estimators. We extend the results of \cite{BTWB} in least-squares density estimation to unbounded dictionaries. Section \ref{sect.EstSelect} presents the results on estimator selection. We extend the algorithm of \cite{BGH11} to least-squares density estimation and obtain a general estimator selection theorem for bounded collections of estimators. We study also our robust approach in this context and extend, for the important subproblem of variable selection the estimator selection theorem to unbounded collections of estimators. Section \ref{sect.Mest} considers the problem of building robust estimators in $M$-estimation. We present a general strategy slightly more complicated than the one of Section \ref{sect.RobEstMean} that we apply to the $M$-estimation frameworks of least-squares density estimation, density estimation with K\"ullback loss and an heteroscedastic and random design regression problem where the target functions are not bounded and the errors not Gaussian in general. Section \ref{section.mixing} presents the extension to the mixing case and the proofs are postponed to the appendix.

\section{Robust estimation of the mean}\label{sect.RobEstMean}
Let $(\Xbf,\Xcal)$ be a measurable space, let $X_{1:n}=(X_{1},\ldots,X_{n)}$ be i.i.d, $\Xbf$-valued random variables with common distribution $P$ and let $X$ be an independent copy of $X_{1}$. Let $V\leq n$ be an integer and let $\B=(B_{1},\ldots,B_{V})$ be a regular partition of $I_{n}\egaldef\set{1,\ldots,n}$, i.e.
\begin{equation}\label{hyp.reg.part}
\hyptag{C_{reg}}\forall K=1,\ldots,V,\qquad \absj{\card{B_{K}}-\frac{n}V}\leq 1\enspace .
\end{equation}
We will moreover always assume, for notational convenience, that $V\leq n/2$ so that, from \eqref{hyp.reg.part}, $\card{B_{K}}\geq n/V-1\geq n/(2V)$.
For all non empty subsets $A\subset I_{n}$, for all real valued measurable functions $t$ (respectively for all integrable functions $t$), we denote by
\[
P_{A}t=\frac1{\card(A)}\sum_{i\in A}t(X_{i}),\qquad Pt=\E\paren{t(X)}\enspace .
\]
For short, we also denote by $P_{n}=P_{I_{n}}$. For all integers $N$, for all $a_{1:N}=(a_{1},\ldots,a_{N})$ in $\R^{N}$, we denote by $\med(a_{1:N})$ any real number $b$ such that
\[
\card\set{i\leq N\telque a_{i}\leq b}\geq \frac N2,\quad\card\set{i=1\leq N\telque a_{i}\geq b}\geq \frac N2\enspace .
\]
Let us finally introduce, for all real valued measurable functions $t$,
\[
\olP_{\B}t=\med\set{P_{B_{K}}t,\;K=1,\ldots,V}.
\]
Our first result is the following concentration inequality.
\begin{proposition}\label{pro.conc.rob}
Let $(\Xbf,\Xcal)$ be a measurable space and let $X_{1:n}$ be i.i.d, $\Xbf$-valued random variables with common distribution $P$. Let $f:\Xbf\mapsto\R$ be a measurable function such that $\var_{P}(f)<\infty$ and let $\delta\in(0,1)$. Let $V\leq n/2$ be an integer and let $\B$ be a partition of $I_{n}$ satisfying \eqref{hyp.reg.part}. If $V\geq \ln(\delta^{-1})$, we have, for some absolute constant $L_{1}\leq 2\sqrt{6e}$, 
\[
\P\set{\olP_{\B} f-Pf> L_{1}\sqrt{\var_{P}f}\sqrt{\frac{V}{n}}}\leq\delta \enspace .
\]
\end{proposition}
\noindent
Proposition \ref{pro.conc.rob} is proved in Section \ref{sect.proof.Prop.conc.rob}.
\begin{remark}
An important feature of this result is that the concentration inequality does not involve the sup-norm of the data. This is the key for the developments presented in Sections \ref{sect.FirstApp} and \ref{sect.EstSelect}. The drawbacks are that the level $\delta$ has to be chosen before the construction of the estimator and that the process $\olP_{\B}$ is not linear.
\end{remark}
\noindent
As a first application, let us give the following control for $\var_{P}f$. 
\begin{corollary}\label{coro.controlVar}
Let $(\Xbf,\Xcal)$ be a measurable space and let $X_{1:n}$ be i.i.d, $\Xbf$-valued random variables with common distribution $P$. Let $f:\Xbf\mapsto\R$ be a measurable function such that $\var_{P}(f^{2})<\infty$ and let $\delta\in(0,1)$. Let $V\leq n/2$ be an integer and let $\B$ be a partition of $I_{n}$ satisfying \eqref{hyp.reg.part}. Assume that $V\geq \ln(\delta^{-1})$ and that
\begin{equation}\label{cond.var}
\hyptag{C(f)} L_{1}\frac{\sqrt{\var_{P}f^{2}}}{Pf^{2}}\sqrt{\frac{V}{n}}\un_{Pf^{2}\neq 0}\leq \frac12.
\end{equation}
Then
\[
\P\set{\var_{P}(f)\leq 2\olP_{\B}f^{2}}\geq1-\delta.
\]
\end{corollary}
\begin{proof}
We can assume that $Pf^{2}\neq 0$, otherwise $\var_{P}(f)=0$ and the proposition is proved. We apply Proposition \ref{pro.conc.rob} to the function $-f^{2}$, since $V\geq \ln(\delta^{-1})$, we have 
\[
\P\set{\olP f^{2}< Pf^{2}\paren{1-L_{1}\frac{\sqrt{\var_{P}f^{2}}}{Pf^{2}}\sqrt{\frac{V}{n}}}}\leq\delta.
\]
We conclude the proof with assumption \eqref{cond.var}.
\end{proof}
\noindent
Using a union bound in Proposition \ref{pro.conc.rob} and Corollary \ref{coro.controlVar}, we get the following corollary.
\begin{corollary}\label{coro.ConcMean}
Let $\Lambda$ be a finite set and let $\pi$ be a probability measure on $\Lambda$. Let $\Dcal=(\psil)_{\lL}$ be a set of measurable functions. Let $(\B_{\lambda})_{\lL}$ be a set of regular partition of $I_{n}$ with $\card(\B_{\lambda})=V_{\lambda}\geq\ln(4(\pi(\lambda)\delta)^{-1})$. Assume that
\begin{equation}\label{hyp.Dic}
\hyptag{C(\Dcal)} \forall \lL,\qquad \var_{P}(\psil^{2})<\infty,\;\mbox{and}\qquad L_{1}\frac{\sqrt{\var_{P}(\psil^{2})}}{P\psil^{2}}\sqrt{\frac{V_{\lambda}}{n}}\un_{P\psil^{2}\neq 0}\leq \frac12
\end{equation} 
Then, for $L_{2}=\sqrt{2}L_{1}$, we have
\[
\P\set{\forall \lL,\qquad \absj{\olP_{\B_{\lambda}} \psil-P\psil}\leq  L_{2}\sqrt{\olP_{\B_{\lambda}}\psil^{2}}\sqrt{\frac{V_{\lambda}}{n}}}\geq1-\delta.
\]
\end{corollary}

\noindent
In Sections \ref{sect.FirstApp} and \ref{sect.EstSelect}, we present several applications of these results. 
\section{Application to Lasso estimators in least-squares density estimation}\label{sect.FirstApp}

Lasso estimators of \cite{Ti96} became popular over the last few years, in particular because they can be computed efficiently in practice. These estimators have been studied in density estimation in \cite{BTWB} for bounded dictionaries. We propose here to revisit and extend some results of \cite{BTWB} with our robust approach. This approach does not require boundedness of the dictionary as will be shown.

Let us recall here the classical framework of density estimation. We observe i.i.d random variables $X_{1:n}$, valued in a measured space $(\Xbf,\Xcal,\mu)$, with common distribution $P$. We assume that $P$ is absolutely continuous with respect to $\mu$ and that the density $\st$ of $P$ with respect to $\mu$ belongs to $L^{2}(\mu)$. We denote respectively by $\psh{.}{.}$ and $\norm{.}$ the inner product and the norm in $L^{2}(\mu)$. The risk of an estimator $\ERM$ of $\st$ is measured by its $L^{2}$-risks, i.e. $\norm{\ERM-\st}^{2}$.
Given a collection $(\ERM_{\theta})_{\theta\in \Theta}$ of estimators (possibly random) of $\st$, we want to select a data-driven $\thetah$ such that 
\[
\P\set{\forall \theta\in\Theta,\;\norm{\st-\ERM_{\thetah}}^{2}\leq C\norm{\st-\ERM_{\theta}}^{2}+R(\theta)}\geq 1-\delta.
\]
In the previous inequality the leading constant $C$ is expected to be close to $1$ and $R(\theta)$ should remain of reasonable size. In that case, $\ERM_{\thetah}$ is said to satisfy an oracle inequality since it behaves as well as an ``oracle'', i.e. a minimizer of $\norm{\st-\ERM_{\theta}}^{2}$ that is unknown in practice. 

Let $\Lambda$ be a finite set of cardinal $M$ and let $\Dcal=\set{\psil,\;\lL}$ be a dictionary, i.e., a finite set of measurable functions. In the case of the Lasso, $\Theta\subset \R^{M}$ and, for all $\theta=(\thetal)_{\lL}\in\Theta$, 
\[
\sthet=\sum_{\lL}\thetal\psil.
\]
The risk of $\sthet$ is equal to 
\begin{align*}
 \norm{\st-\sthet}^{2}=&\norm{\st}^{2}+\norm{\sthet}^{2}-2\int \st\sthet d\mu\\
 =&\norm{\st}^{2}+\norm{\sthet}^{2}-2\sum_{\lL}\thetal\int \st\psil d\mu.
\end{align*}
Therefore, the best estimator, or the oracle, in the collection $(\sthet)_{\theta\in\Theta\subset\R^{M}}$ minimizes the ideal criterion
\[
\critid(\theta)=\norm{\sthet}^{2}-2\sum_{\lL}\thetal P\psil.
\]
The idea of \cite{BTWB} is to replace the unknown $\critid$ by the following penalized empirical version
\[
\crit_{Las}(\theta)=\norm{\sthet}^{2}-2\sum_{\lL}\thetal P_{n}\psil+2\sum_{\lL}\omegal\absj{\thetal} \enspace,
\]
for proper choice of the weights $\omegal$. We modify this Lasso-type criterion with our robust version. Let $\B$ be a regular partition of $\set{1,\ldots,n}$ with cardinality $V$ to be defined later. Let $\olP_{\B}$ be the associated empirical process introduced in Proposition \ref{pro.conc.rob}. Our criterion is given by
\[
\crit(\theta,\B)=\norm{\sthet}^{2}-2\sum_{\lL}\thetal \olP_{\B}\psil+2\sum_{\lL}\omegal\absj{\thetal}\enspace ,
\]
for weights $\omegal$ to be defined later. Our final estimator is then given by $s_{\thetah}$, where
\[
\thetah=\arg\min_{\theta\in\Theta}\set{\crit(\theta,\B)}.
\]
For all $\theta\in \R^{M}$, let
\(
J(\theta)=\set{\lL,\;\thetal\neq0},\;M(\theta)=\card{J(\theta)}.
\)
Let $\delta\in(0,1)$ and, for every $\llpL,$ let
\begin{align*}
 \rho_{M}(\lambda,\lp)=&\frac{\psh{\psil}{\psilp}}{\norm{\psil}\norm{\psilp}},\qquad \rho(\theta)=\max_{\lambda\in J(\theta)}\max_{\lp\neq\lambda}\absj{\rho_{M}(\lambda,\lp)}\enspace,\\
 \rho_{*}(\theta)=&\sum_{\lambda\in J(\theta)}\sum_{\lp>\lambda}\absj{\rho_{M}(\lambda,\lp)},\quad  G(\theta)=\sum_{\lambda\in J(\theta)}\omegal^{2} \enspace , \\
 F(\theta)=&\sqrt{\frac{n}{\ln\paren{\frac{2M}{\delta}}}}\max_{\lambda\in J(\theta)}\set{\frac{\omegal}{\norm{\psil}}},\quad G=\sqrt{\frac{\ln\paren{\frac{2M}{\delta}}}{n}}\max_{\lL}\set{\frac{\norm{\psil}}{\omegal}}\enspace.
\end{align*}
Let us call $\Gamma_{M}$ the gram matrix of $\Dcal$, i.e., the matrix with entries $\rho_{M}(\lambda,\lp)$ and let $\zeta_{M}$ be the smallest eigenvalue of $\Gamma_{M}$.
The following assumptions were used in \cite{BTWB} to state the results.
\begin{equation}
\label{Cond.H1beta}\tag{\ensuremath{\mathbf{H_{1}(\theta)}}} 16GF(\theta)M(\theta)\leq 1\enspace .
\end{equation}
\begin{equation}
\label{Cond.H2beta}\tag{\ensuremath{\mathbf{H_{2}(\theta)}}} 16GF(\theta)\rho_{*}(\theta)\sqrt{M(\theta)}\leq 1\enspace .
\end{equation}
\begin{equation}
\label{Cond.H3beta}\tag{\ensuremath{\mathbf{H_{3}(\theta)}}} \Gamma_{M}>0\;\mbox{and}\;\zeta_{M}\geq \kappa_{M}>0\enspace .
\end{equation}
This estimator satisfies the following Theorem.
\begin{theorem}\label{theo.LassoDensity}
Let $\delta\in (0,1)$, let $\Dcal=(\psil)_{\lL}$, let $\B$ be a regular partition of $\set{1,\ldots,n}$, with $V\geq\ln(4M\delta^{-1})$. Assume that \eqref{hyp.Dic} holds for $\pi$ the uniform distribution on $\Lambda$ and, for all $\lambda\in \Lambda$, $V_{\lambda}=V$. If \eqref{Cond.H1beta},  \eqref{Cond.H2beta} or \eqref{Cond.H3beta} hold, the estimator $s_{\thetah}$ defined with $L_{3}=2L_{2}$ and
\[
\omegal \geq L_{3}\sqrt{\olP_{\B}\psil^{2}}\sqrt{\frac{V}{n}}
\]
satisfies, for all $\alpha>1$, with probability larger than $1-2\delta $, $\forall\theta\in\Theta$,
\[
\norm{s_{\thetah}-\st}^{2}+\frac{\alpha}{2(\alpha+1)}\sum_{\lL}\omegal\absj{\thetal-\thetahl}\leq  \frac{\alpha+1}{\alpha-1}\norm{\sthet-\st}^{2}+\frac{8\alpha^{2}}{\alpha-1}R(\theta) \enspace.
\]
Where the remainder term $R(\theta)$ is equal to $F^{2}(\theta)M(\theta)\ln(4M\delta^{-1})/n$ under \eqref{Cond.H1beta} or \eqref{Cond.H2beta} and to $G(\theta)/(n\kappa_{M})$ under \eqref{Cond.H3beta}.
\end{theorem}

\begin{remark}
The proof of Theorem \ref{theo.LassoDensity} is decomposed into two propositions. The main one, Proposition \ref{pro.BTWB} below, follows from the proofs of Theorems 1, 2 and 3 of \cite{BTWB} that won't be reproduced here. We refer to this paper for the proof and for further comments on the main theorem. Let us remark that the improvement that we get using our robust approach is that we only require $P\psil^{4}<\infty$ in our results whereas in \cite{BTWB}, it is required that $\norm{\psil}_{\infty}<\infty$.
\end{remark}

\begin{remark}
An interesting feature of our result is that it allows to revisit famous procedures built with the empirical process. $P_{n}$ has to be replaced by $\olP_{\B}$ for a proper choice of $V$ and Bennett's, Bernstein's or Hoeffding's inequalities can be replaced by Proposition \ref{pro.conc.rob}. Theorem \ref{theo.LassoDensity} is just an example of this general principle.
\end{remark}

\begin{proposition}\label{pro.BTWB}
Under assumptions \eqref{Cond.H1beta},  \eqref{Cond.H2beta} or \eqref{Cond.H3beta}, the following condition
\begin{equation}\label{Cond.Suff.Lasso}
\forall \theta\in\Theta,\quad \norm{s_{\thetah}-\st}^{2}+\sum_{\lL}\omegal\absj{\thetal-\thetahl}\leq  \norm{\sthet-\st}^{2}+4\sum_{\lambda\in J(\theta)}\omegal\absj{\thetal-\thetahl},
\end{equation}
implies, for all $\alpha>1$, $\forall\theta\in\Theta$,
\[
\norm{s_{\thetah}-\st}^{2}+\frac{\alpha}{2(\alpha+1)}\sum_{\lL}\omegal\absj{\thetal-\thetahl}\leq  \frac{\alpha+1}{\alpha-1}\norm{\sthet-\st}^{2}+\frac{8\alpha^{2}}{\alpha-1}R(\theta)\enspace .
\]
\end{proposition}

\noindent
In order to prove Theorem \ref{theo.LassoDensity}, we only have to ensure that Condition \eqref{Cond.Suff.Lasso} holds with the required probability, which will be done in the following proposition.

\begin{proposition}\label{pro.interm.Lasso}
Let $\delta\in (0,1)$. Let $\B$ be a regular partition of $\set{1,\ldots,n}$, with $V\geq \ln(4M\delta^{-1})$ and assume that \eqref{hyp.Dic} holds. Assume that, 
\begin{equation}\label{cond.weight.lasso}
\forall \lL,\;\omegal\geq L_{3}\sqrt{\olP_{\B}\psil^{2}}\sqrt{\frac{V}{n}}\enspace .
\end{equation}
Then, with probability larger than $1-\delta$, $\forall \theta\in\Theta$,
\[
\norm{s_{\thetah}-\st}^{2}+\sum_{\lL}\omegal\absj{\thetal-\thetahl}\leq  \norm{\sthet-\st}^{2}+4\sum_{\lambda\in J(\theta)}\omegal\absj{\thetal-\thetahl}\enspace.
\]
\end{proposition}

\noindent 
Proposition \ref{pro.interm.Lasso} is proved in Section \ref{sect.Proof.Prop.Interm.Lasso}.

\section{Estimator Selection}\label{sect.EstSelect}
Estimator selection is a new theory developed in \cite{Ba11,BGH11}. It covers important statistical problems as Model Selection, aggregation of estimators, selecting tuning constants in statistical procedures and it allows to compare several estimation procedures. \cite{Ba11} developed a general approach in the spirit of \cite{Bi06}. It applies in various frameworks but it does not provide a method for practitioner because the resulting estimators are too hard to compute. On the other hand, \cite{BGH11} worked in a Gaussian regression framework and obtained efficient estimators. The approach of \cite{BGH11} can be adapted to the density estimation framework as we will see. In order to keep the paper of reasonable size, we won't give practical ways to define the collections of estimators. This fundamental issue and several others are extensively discussed in \cite{BGH11}, we refer to this paper for all the practical consequences of  the main oracle inequality. As \cite{BGH11} worked in a Gaussian framework, they do not require any $L^{\infty}$-norm. In order to emphasize the advantages of robust empirical mean estimators in this problem, let us first extend the results of \cite{BGH11} to the density estimation framework where conditions on $L^{\infty}$-norms are classical \cite{BM97,BBM99,Le09}. 

\subsection{The empirical version}

Let $(\ERM_{\theta})_{\theta\in \Theta}$ be a collection of estimator of $\st$. Let $(S_{m})_{\mM}$ be a collection of linear subspaces of measurable functions and for all $\theta\in \Theta$, let $\M_{\theta}$ be a subset of $\M$, possibly random. For all $\theta\in \Theta$ and all $\mM_{\theta}$, we choose an estimator $\ERMP_{\theta,m}$, for example, the orthogonal projection of $\ERM_{\theta}$ onto $S_{m}$. Finally, we denote by $\pen:\M\rightarrow \R^{+}$ a function to be defined later. Let $\alpha>0$ and let
\begin{align}
\notag\crit_{\alpha}(\theta)=\min_{\mM_{\theta}}&\set{\norm{\ERMP_{\theta,m}}^{2}-2P_{n}\ERMP_{\theta,m}+\alpha\norm{\ERMP_{\theta,m}-\ERM_{\theta}}^{2}+\pen(m)},\\
\label{def.EstSelect}
&\qquad \qquad\thetah=\arg\min_{\theta\in\Theta}\crit_{\alpha}(\theta)\enspace .
\end{align}
Let $\Pcal(\M)$ be the set of probability measures on $\M$, let $\Boule(m)$ be the unit ball in $L^{2}$-norm of $S_{m}$, $\Boule(m)=\set{t\in S_{m},\;\norm{t}\leq 1}$ and let $\Psi_{m}=\sup_{t\in\Boule(m)}t^{2}$. Let $m_{o}$ be a minimizer of $n\norm{s-s_{m}}^{2}+P\Psi_{m}$. Let us introduce the following assumption.
\begin{align}\label{Hyp.Sel.Est.Class}
\notag\exists \pi\in\Pcal(\M),\;\varepsilon_{n}\flens 0,\;\delta_{o}\in (0,1)&\;\telque\\
\hyptag{CSED}\forall \mM,&\quad\frac{\norm{s_{m}-s_{m_{o}}}_{\infty}\ln\paren{\frac4{\pi(m)\delta_{o}}}}{n\norm{s-s_{m}}^{2}+P\Psi_{m}}\leq \varepsilon_{n}.
\end{align}
The following result holds.
\begin{theorem}\label{theo.Estim.Select.Classic}
Let $X_{1:n}$ be i.i.d, $\Xbf$-valued, random variables with common density $\st\in L^{2}(\mu)$. For all $\mM$, let $s_{m}$ be the orthogonal projection of $\st$ onto $S_{m}$ and assume that \eqref{Hyp.Sel.Est.Class} holds. Let $\delta\geq \delta_{o}$ and let $\varepsilon^{\prime}_{n}=4\sqrt{\varepsilon_{n}}+\varepsilon_{n}/3$ and let $n_{o}$ such that, for all $n\geq n_{o}$, $\varepsilon_{n}^{\prime}\leq 1/2$.
Let also 
\[
r_{m}(\delta)=\frac{\sqrt{\norm{\Psi_{m}}_{\infty}}\ln\paren{\frac{2}{\pi(m)\delta}}}{n}\enspace .
\]
Let $\thetah$ be the estimator \eqref{def.EstSelect} selected by a penalty $\pen$ such that, for some $\nu\in (0,1)$,
\[
\pen(m)\geq \paren{\frac52+2L_{0}\nu}\frac{P\Psi_{m}}n+\frac{2L_{o}\norm{s}}{\nu}r_{m}(\delta)+\frac{2L_{o}}{\nu^{3}}r^{2}_{m}(\delta)\enspace ,
\]
where $L_{o}\leq 16(\ln 2)^{-1}+8$. There exists a constant $L_{\alpha}$ such that, for all $n\geq n_{o}$, with probability larger than $1-\delta$, $\forall \theta\in\Theta,\;$
\begin{equation}\label{eq.or.est.select.density}
L_{\alpha}\norm{\ERM_{\thetah}-\st}^{2}\leq \norm{\ERM_{\theta}-\st}^{2}+\min_{\mM_{\theta}}\set{\norm{\ERMP_{\theta,m}-\ERM_{\theta}}^{2}+2\pen(m)}\enspace .
\end{equation}
\end{theorem}
\noindent
Theorem \ref{theo.Estim.Select.Classic} is proved in Section \ref{Proof.theo.Estim.Select.Classic}.
\begin{remark}
It is shown in \cite{Le10} that \eqref{Hyp.Sel.Est.Class} typically holds in classical collections of models as Fourier spaces, and growing wavelet or histogram spaces, under reasonable assumptions on the risk of the estimator, for $\delta_{o}=O(n^{-\kappa})$. We refer to this paper for further details on this assumptions. 
\end{remark}
\begin{remark}\label{rem:firstUseRob}
It is usually assumed that, for some constant $\Gamma$, $\norm{\Psi_{m}}_{\infty}\leq \Gamma P\Psi_{m}$. In that case, for those $m$ such that $P\Psi_{m}\rightarrow \infty$, we have $r_{m}(\delta)=o\paren{n^{-1}P\Psi_{m}}$ so that the condition on the penalty is asymptotically satisfied as soon as $\pen(m)>2.5P\Psi_{m}/n$. This last condition holds if we choose $\pen(m)>2.5\norm{\Psi_{m}}_{\infty}/n$ for example. A first application of our robust approach is that, using Corollary \ref{coro.ConcMean}, we can choose $\pen(m)>5\olP_{\B}\Psi_{m}/n$ under the more reasonable assumptions \eqref{hyp.Dic} with $\Dcal=(\Psi_{m})_{\mM}$. In that case, the condition on the penalty only holds with large probability but the interested reader can checked that this will only change $\delta$ into $2\delta$ in \eqref{eq.or.est.select.density}.
\end{remark}
\begin{remark}
This result includes as a special case classical model selection framework of \cite{BM97, BBM99} where $\Theta=\M$ and, for all $\mM$, $\ERM_{m}$ is the projection estimator onto $S_{m}$.
\end{remark}
\begin{remark}
It covers also the problem of choosing a tuning parameter in a statistical estimation method. In that case $\Theta$ is usually a subset of $\R$ and $\ERM_{\theta}$ is the estimator selected by the statistical method, with the tuning parameter equal to $\theta$. It allows also to mix several estimation strategies, in that case $\Theta$ is typically the product of a finite set $A$ describing the set of methods with a subspace of $\R$ or $\R^{d}$ describing the possible values of the tuning parameters (see \cite{BGH11}).
\end{remark}
\subsection{A robust version}
We have already shown in remark \ref{rem:firstUseRob} that our robust approach can be used to build the penalty term in the ``classical'' estimator selection described above. However, this first approach relies on the assumptions that $\norm{\Psi_{m}}_{\infty}\leq \Gamma P\Psi_{m}$, $\norm{s_{m}-s_{m^{\prime}}}_{\infty}\leq \Gamma\norm{s_{m}-s_{m^{\prime}}}$. We will now present another approach, totally based on robust estimators, which works without these assumptions. In this section, $(\psil)_{\lL}$ is an orthonormal system in $L^{2}(\mu)$ and $\Lambda_{M}\subset \Lambda$ is a finite subset. Let $(\Lambda_{m})_{\mM}$ be a collection of subsets of $\Lambda_{M}$, let $\Lambda_{n}=\cup_{\mM}\Lambda_{m}$, and for all $\mM$, let $S_{m}$ the linear span of $(\psil)_{\lL_{m}}$.  Let $(\ERM_{\theta})_{\theta\in \Theta}$ be a collection of estimator of $\st$. For all $\theta\in \Theta$, let $\M_{\theta}$ be a subset of $\M$, possibly random. For all $\theta\in \Theta$ and all $\mM_{\theta}$, we choose an estimator $\ERMP_{\theta,m}$, for example, the orthogonal projection of $\ERM_{\theta}$ onto $S_{m}$. For all $\theta\in \Theta$, for all $\mM_{\theta}$ and for all $\lL_{m}$, let $\betahl^{\theta}=\psh{\ERMP_{\theta,m}}{\psil}$. For all $\lL_{n}$, let $\B_{\lambda}$ be a regular partition of $\set{1,\ldots,n}$, with cardinality $V_{\lambda}$ to be defined later. Let $\pen:\M\rightarrow \R^{+}$ a function to be defined later. Let $\alpha>0$ and let
\begin{align}
\notag\crit_{\alpha}(\theta)=\min_{\mM_{\theta}}&\set{\norm{\ERMP_{\theta,m}}^{2}-2\sum_{\lL_{m}}\betahl^{\theta}\olP_{\B_{\lambda}}\psil+\alpha\norm{\ERMP_{\theta,m}-\ERM_{\theta}}^{2}+\pen(m)}\\
\label{def.EstSelectRob}&\qquad \qquad\thetah=\arg\min_{\theta\in\Theta}\crit_{\alpha}(\theta)\enspace .
\end{align}
This estimator satisfies the following theorem.

\begin{theorem}\label{theo.EstSelRob}
Let $X_{1:n}$ be i.i.d, $\Xbf$-valued random variables with common density $\st\in L^{2}(\mu)$. Let $\delta\in(0,1)$ and $\epsilon\in(0,1/4)$. Let $(\ERM_{\theta})_{\theta\in\Theta}$ be a collection of estimators, let $(\psil)_{\lL}$, $\Lambda_{M}$, $(\Lambda_{m})_{\mM}$, $(S_{m})_{\mM}$,
 $(\M_{\theta})_{\theta\in\Theta}$ and $(\ERMP_{\theta,m})_{\theta\in\Theta,\mM_{\theta}}$ be defined as above. Let $\pi$ be a probability measure on $\Lambda_{M}$. Assume that, for all $\lambda\in \Lambda_{M}$, $V_{\lambda}\geq\ln(2(\pi(\lambda)\delta)^{-1})$
%  and that \eqref{hyp.Dic} holds for $\Dcal=(\psil)_{\lL_{M}}$. 
Then, the estimator $\thetah$ defined in \eqref{def.EstSelectRob} with $\pen(m)\geq \frac{L_{4}}{\epsilon n}\sum_{\lL_{m}}\var_{P}(\psil)V_{\lambda}$, where $L_{4}=9L_{1}^{2}/4$ satisfies, with probability larger than $1-\delta/2$, $\forall \theta\in\Theta,\; $
\[
\frac{(1-4\epsilon)\wedge \alpha}{2(8+\alpha)}\norm{\ERM_{\thetah}-\st}^{2}\leq \norm{\ERM_{\theta}-\st}^{2}+\min_{\mM_{\theta}}\set{\norm{\ERM_{\theta}-\ERMP_{\theta}}^{2}+\pen(m)}\enspace.
\]
\end{theorem}
\noindent
Theorem \ref{theo.EstSelRob} is proved in Section \ref{sect.proof.theo.Estim.select.Rob}.
\begin{remark}
In order to choose the penalty term, one can use Corollary \ref{coro.ConcMean}. Under assumption \eqref{hyp.Dic} for the dictionary $(\psil)_{\lL_{M}}$, we have, with our choice of $V_{\lambda}$, 
\[
\P\set{\forall \lL_{M},\;\var_{P}(\psil)\leq 2\olP_{\B_{\lambda}}\psil^{2}}\geq 1-\frac{\delta}2\enspace .
\] 
we can therefore use the data-driven penalty 
\[
\pen(m)=\frac{2L_{4}}{\epsilon n}\sum_{\lL_{m}}(\olP_{\B_{\lambda}}\psil^{2})V_{\lambda}\enspace .
\]
\end{remark}
\begin{remark}
Compared to Theorem \ref{theo.Estim.Select.Classic}, we see that the collection of models $(S_{m})_{\mM}$ is restricted here to a family generated by an orthonormal system $(\psil)_{\lL_{M}}$, the penalty term is in general heavier, which yields a loss (of order $\sup_{\lL_{m}}V_{\lambda}$) in the convergence rates. On the other hand, we do not require $\norm{\Psi_{m}}_{\infty}$ to be finite, we only need a finite moment of order $4$ for the functions $\psil$. Moreover, in order to build a data-driven penalty in Theorem \ref{theo.Estim.Select.Classic}, we asked that $\norm{\Psi_{m}}_{\infty}<< (P\Psi_{m})^{2}$, whereas we only require now a bound on $\sqrt{P\psil^{4}}/P\psil^{2}\un_{P\psil^{2}\neq 0}$ of order smaller than $\sqrt{n/V_{\lambda}}$. In order to emphasize the difference between these conditions, let us consider the case of an histogram, where for some disjoint measurable sets $(I_{\lambda})_{\lL_{m}}$, $\psil=(\mu(I_{\lambda}))^{-1/2}\un_{I_{\lambda}}$. In that case, we have
\[
\norm{\Psi_{m}}_{\infty}=\max_{\lL_{m}}\frac1{\mu(I_{\lambda})},\;P\psil^{2}=\frac{PI_{\lambda}}{\mu(\psil)},\;\sqrt{P\psil^{4}}=\frac{\sqrt{PI_{\lambda}}}{\mu(I_{\lambda})}\enspace.
\]
If $\st$ is upper bounded by $c_{+}$, we deduce that $P\Psi_{m}\leq c_{+}d_{m}$, where $d_{m}$ is the number of pieces of the histogram. If, moreover, $\st$ is lower bounded by $c_{-}$, we have 
\[
\frac{\sqrt{P\psil^{4}}}{P\psil^{2}}\leq \frac{1}{c_{-}\sqrt{\mu(I_{\lambda})}}.
\]
The condition of Theorem \ref{theo.EstSelRob} is therefore satisfied as soon as, for some constant $r$ sufficiently large, $\mu(I_{\lambda})\geq r^{-1}nV_{\lambda}^{-1}$ whereas the condition of Theorem \ref{theo.Estim.Select.Classic} holds only if $\mu(I_{\lambda})>>d_{m}^{-2}$. This last condition is much more restrictive when the histograms are irregular.
\end{remark}
\begin{remark}
Important collections $(\psil)_{\lL}$ are the wavelet spaces, used for example in \cite{DJKP}. It is shown for example in \cite{BM97} that the hard thresholded estimator of \cite{DJKP} coincide with the estimator chosen by model selection with the penalty $d_{m}\ell(n)$, where $\ell(n)$ is the threshold. Moreover, the soft thresholded estimator coincide with the Lasso-estimator presented in the previous section (see for example \cite{BTWB} for details). Our estimator selection procedure can be used with wavelet estimators; it allows to select with the data the best strategie, together with the best thresholds.
\end{remark}

\section{Robust Estimators in $M$-estimation}\label{sect.Mest}
The rest of the paper is devoted to the study of robust estimators in a more general context of $M$-estimation. These estimators will be defined using a slightly more elaborated construction than the one presented in Section \ref{sect.RobEstMean}. This general principle will then be applied to classical problems as density estimation and regression.
\subsection{The general case}
Hereafter, $(\Xbf,\Xcal)$ denotes a measurable space, $\gamma:\Xbf\rightarrow \R$ is a measurable function, called contrast, $P$ be a probability measure, we want to estimate the target $\st=\arg\min_{t\in \Xcal}P\gamma(t)$ based on the observation of i.i.d random variables $X_{1:n}=X_{1},\ldots,X_{n}$ with common distribution $P$. Let $8\leq V\leq n/2$ and let $\B$ be a a regular partition of $\set{1,...,n}$, with cardinality $V$. Let $X_{B_{K}}=(X_{i})_{i\in B_{K}}$ and let $S$ be a subspace of $\Xbf$. Let $\ERM_{K}$ be any estimator defined as a function $\ERM_{K}=F(X_{B_{K}})$ valued in $S$. For all $K,K'=1,\ldots,V$ and for all functions $t$, let
\[\olP_{K,K'}t=\med\set{P_{B_{J}}t,\;J\neq K,K'}.
\]
Our final estimator is defined as
\begin{equation}\label{def:Estimator}
\ERM=\ERM_{\Ks},\quad\mbox{where}\; \Ks=\arg\min_{K=1,...,V}\max_{K'=1,\ldots,V}\set{\olP_{K,K'}\paren{\gamma(\ERM_{K})-\gamma(\ERM_{K'})}}\enspace.
\end{equation}
The risk of an estimator $\ERM$ is measured with the excess risk 
\[
\ell(\ERM,s_{\star})=P\paren{\gamma(\ERM)-\gamma(s_{\star})} \enspace .
\]
We denote by $s_{o}=\inf_{t\in S}P\gamma(t)$. We assume the following margin type condition.
\begin{align}\label{Cond.marg}
\notag \exists N\in \N, (\alpha_{i})_{i=1,\ldots,N}< 1,&\;(\sigma_{i}^{2})_{i=0,\ldots,N}<\infty,\; \epsilon\leq 1,\Omega_{\epsilon}\;\mbox{with}\;\P\paren{\Omega_{\epsilon}}\geq 1-\epsilon,\\
\notag\telque\mbox{on }\Omega_{\epsilon},\quad\max_{K}&\norm{\var\croch{\paren{\gamma(\ERM_{K})-\gamma(s_{o})}(X)\sachant X_{B_{K}}}}_{\infty}\\
\hyptag{CMarg}&\leq  \sigma_{0}^{2}\ell(\ERM_{K},s_{o})^{2}+\sum_{i=1}^{N}\sigma_{i}^{2}\ell(\ERM_{K},s_{o})^{2\alpha_{i}}.
%\tag{\ensuremath{\mathbf{C2}}} \forall K\neq J=1,\ldots, V,\; \forall x>0, &\qquad \P\paren{\absj{(P_{J}-P)\paren{\gamma(\ERM_{K})}-\gamma(s_{o})}>x}\leq g(x).\\
%\tag{\ensuremath{\mathbf{C3}}} &\qquad\var\paren{(\gamma(s_{o})-\gamma(\st)(X)}=\sigma^{2}<\infty.
\end{align}
\begin{remark}
Assumption \eqref{Cond.marg} is not classical and might surprise at first sight. It will be discussed in Section \ref{sect.ApplMEstRob}. In particular, we show in this section that \eqref{Cond.marg} holds under few assumptions on the data and the space $S$ in density estimation and in a general not bounded, heteroscedastic, non Gaussian and random design regression framework.
\end{remark}
We finally denote, for all real numbers $a$ by $\PESup{a}$ the smallest integer $b$ such that $b\geq a$. Our result is the following.
\begin{theorem}\label{theo:FundamentalResult}
Let $X_{1:n}$ be i.i.d random variables and let $\delta>0$ such that $\PESup{\ln(\delta^{-2})}\leq n/2$. Let $\B$ be a regular partition of $\set{1,\ldots,n}$, with $V=\PESup{\ln(\delta^{-2})}\vee 8$. Let $(\ERM_{K})_{K=1,\ldots,V}$ be a sequence of estimators satisfying \eqref{Cond.marg} and let $\ERM_{\Ks}$ be the associated estimator defined in (\ref{def:Estimator}).
Let $C_{0}=L_{1}$ and for all $i=1,\ldots,N$, let $C_{i}=4(1-\alpha_{i})(L_{1}\alpha_{i}^{\alpha_{i}})^{1/(1-\alpha_{i})}$.
For all $\Delta>1$, let 
\[
\nu_{n}(\Delta)=C_{0}\sigma_{0}\sqrt{\frac{V}{n}}+\frac N{\Delta},\; R_{n}(\Delta)=\sum_{i=1}^{N}C_{i}\paren{\Delta^{\alpha_{i}}\sigma_{i}\sqrt{\frac{V}{n}}}^{\frac1{1-\alpha_{i}}}.
\]
For all $\Delta>1$, with probability larger than $1-\delta-\epsilon$,
\[
\paren{1-\nu_{n}(\Delta)}\ell(\ERM_{\Ks},s_{o})\leq  \paren{1+3\nu_{n}(\Delta)}\inf_{K}\set{\ell(\ERM_{K},s_{o})}+R_{n}(\Delta)\enspace.
\]
In particular, for all $\Delta$, $n$ such that $\nu_{n}(\Delta)\leq 1/2$, we have, with probability larger than $1-\epsilon-\delta$,
\[
\ell(\ERM_{\Ks},\st)\leq  \ell(s_{o},\st)+\paren{1+8\nu_{n}(\Delta)}\inf_{K}\set{\ell(\ERM_{K},s_{o})}+2R_{n}(\Delta).
\]
\end{theorem}
\noindent
Theorem \ref{theo:FundamentalResult} is proved in Section \ref{Sect.proof.fundresult}.

\begin{remark}
Assume that, for all $K$, $\ERM_{K}=\arg\min_{t\in S} P_{B_{K}}t$ and let $\ERM_{n}$ denote the classical empirical risk minimizer, $\ERM_{n}=\arg\min_{t\in S}P_{n}\gamma(t)$. $\ERM_{n}$ is known to have a nice behavior when the contrast $\gamma$ is bounded and under some margin condition such as (see \cite{MT99, MN06} and the references therein),
\[
\exists \alpha\in [0,1]\;\mbox{and}\; A\geq 1;\telque \forall t\in S,\;\var_{P}((\gamma(t)-\gamma(s_{o})))\leq A\ell(t,s_{o})^{\alpha}\enspace . 
\]
Our approach does not require a finite sup norm for $\gamma$. However, it should be mentioned that the confidence level $\delta$ has to be chosen in advance and that we loose a log factor in the expectation, as the $\ERM_{K}$ are built with only $V/n$ data.
\end{remark}
\subsection{Application to classical statistical problems}\label{sect.ApplMEstRob}

The aim of this section is to apply Theorem \ref{theo:FundamentalResult} in some well known problems. We show in particular that Condition \eqref{Cond.marg} holds in these frameworks.

\subsubsection{Density estimation with $L^{2}$-loss}
Assume that $X_{1:n}$ have common marginal density $\st$. Assume that $\st\in L^{2}(\mu)$. Then we have, for all $t$ in $L^{2}(\mu)$,
\[
\norm{\st-t}^{2}=\norm{\st}^{2}+\norm{t}^{2}-2\int tsd\mu=\norm{\st}^{2}+\norm{t}^{2}-2Pt=\norm{\st}^{2}+P\gamma(t).
\]
In the previous inequality,
\(
\gamma(t)=\norm{t}^{2}-2t,\) thus, \(\st=\arg\min_{t\in L^{2}(\mu)} P\gamma(t).
\)
We also have 
\[
P\gamma(\st)=\norm{\st}^{2}-2\int \st^{2}d\mu=-\norm{\st}^{2}\enspace,
\]
hence,
\[
\ell(\ERM,\st)=\norm{t}^{2}-2\int tsd\mu+\norm{\st}^{2}=\norm{\ERM-\st}^{2}\enspace.
\]
Let $S$ a linear space of measurable functions and let 
\[
\ERM_{K}=\arg\min_{t\in S}P_{B_{K}}\gamma(t).
\]
$\ERM_{K}$ is easily computed since, for all orthonormal bases $(\psil)_{\lL}$ of $S$, we have
\[
\ERM_{K}=\sum_{\lL}(P_{B_{K}}\psil)\psil.
\]
Let us denote by $s_{o}$ the orthogonal projection of $\st$ onto $S$. We have, from Pythagoras relation,
\[
\ell(\ERM_{K},\st)=\norm{\ERM_{K}-\st}^{2}=\norm{\ERM_{K}-s_{o}}^{2}+\norm{s_{o}-\st}^{2}.
\]
Given an orthonormal basis $(\psil)_{\lL}$ of $S$, we have
\begin{align*}
 \E\paren{\norm{\ERM_{K}-s_{o}}^{2}}&=\E\paren{\sum_{\lL}\paren{(P_{B_{K}}-P)\psil}^{2}}\\
 &=\frac{\sum_{\lL}\var(\psil(X_{1}))}{|B_{K}|}=\frac{P\Psi-\norm{s_{o}}^{2}}{|B_{K}|}\enspace.
\end{align*}
In the previous inequality, the function $\Psi$ is equal to
\[
\Psi=\sum_{\lL}\psil^{2}=\sup_{t\in B}t^{2},\qquad \mbox{where}\;B=\set{t\in S,\;\norm{t}\leq 1}.
\]
The following proposition holds.
\begin{proposition}\label{prop.CondMarg.Denisty}
Let $X_{1:n}$ be random variables with common marginal density $\st$ with respect to the Lebesgue measure $\mu$. Assume that $\st\in L^{2}(\mu)$. Let $S$ be a linear space of measurable functions. Let $V$ be an integer and let $B_{1},\ldots,B_{V}$ be a regular partition of $\set{1,\ldots,n}$. For all $K=1,\ldots,V$, let $\ERM_{K}$ be any estimator taking value in $S$ and measurable with respect to $\sigma(X_{B_{K}})$. Then Condition \eqref{Cond.marg} is satisfied with $N=1$, $\sigma_{0}=0$, $\alpha_{1}=1/2$, $\sigma_{1}^{2}=P\Psi-\norm{s_{o}}^{2}$. 
\end{proposition}
\begin{proof}
We have $\gamma(\ERM_{K})=\norm{\ERM_{K}}^{2}-2\ERM_{K}$, hence,
\begin{align*}
\var\paren{\gamma(\ERM_{K}(X))-\gamma(s_{o})\sachant X_{B_{K}}}=4\var\paren{(\ERM_{K}-s_{o})(X)\sachant X_{B_{K}}}.
\end{align*}
We can write $\ERM_{K}=\sum_{\lL}\betlassol^{K} \psil$, with $\betlassol^{K}$ measurable with respect to $\sigma(X_{B_{K}})$.
From Cauchy-Schwarz inequality, using the independence between $X$ and the $X_{i}$,
\begin{align*}
\var\left((\ERM_{K}-s_{o})(X)\right.&\left.\sachant X_{B_{K}}\right)\\
&=\E\paren{\paren{\sum_{\lL}(\betlassol^{K}-P\psil)(\psil(X)-P\psil)}^{2}\sachant X_{B_{K}}}\\
&\leq \sum_{\lL}[\betlassol^{K}-P\psil]^{2}\E\paren{\sum_{\lL}[\psil(X)-P\psil]^{2}}\\
&=\norm{\ERM_{K}-s_{o}}^{2}\paren{P\Psi-\norm{s_{o}}^{2}}.
\end{align*}
Therefore, \eqref{Cond.marg} holds with $N=1$, $\sigma_{0}=0$, $\alpha_{1}=1/2$, $\sigma_{1}^{2}=P\Psi-\norm{s_{o}}^{2}$. 
\end{proof}

\noindent
We can deduce from Theorem \ref{theo:FundamentalResult} the following result.
\begin{proposition}\label{prop:DensityL2}
Let $X_{1:n}$ be i.i.d random variables with common density $\st$ with respect to the Lebesgue measure $\mu$. Assume that $\st\in L^{2}(\mu)$. Let $\delta>0$ such that $\PESup{\ln(\delta^{-2})}\leq n/2$. Let $\B$ be a regular partition of $\set{1,\ldots,n}$, with $V=\PESup{\ln(\delta^{-2})}\vee 8$. $\forall K=1,\ldots,V$, let $\ERM_{K}=\arg\min_{t\in S}P_{B_{K}}\gamma(t)$ and let $\ERM_{\Ks}$ be the associated estimator defined in (\ref{def:Estimator}). We have, for $L_{5}=2\sqrt{e}+8L_{1}e^{1/4}$,
\[
\P\paren{\norm{\ERM_{\Ks}-\st}^{2}>\norm{s_{o}-\st}^{2}+L_{5}\paren{P\Psi-\norm{s_{o}}^{2}}\frac{V}n}\leq 2\delta\enspace .
\]
\end{proposition}

\begin{remark}
A classical density estimator is the minimizer of the empirical risk, 
\[
\ERM_{m}=\arg\min_{t\in S_{m}}P_{n}\gamma(t)\enspace .
\]
This estimator satisfies the following risk bound (see for example \cite{Le09}), with probability larger than $1-\delta$, for all $\epsilon>0$,
\[
\norm{\ERM-\st}^{2}\leq \norm{s_{o}-\st}^{2}+(1+\epsilon)\frac{P\Psi-\norm{s_{o}}^{2}}n+C^{\prime}\paren{\frac{v^{2}\ln(\delta^{-1})}{n\epsilon}+\frac{b^{2}\ln(\delta^{-1})^{2}}{\epsilon^{3}n^{2}}}.
\]
In the last inequality $v^{2}=\sup_{t\in B}P((t-Pt)^{2})\leq P\Psi-\norm{s_{o}}^{2}$, $b^{2}=\norm{\Psi}_{\infty}$. The empirical risk minimizer is better when the bound $b^{2}\leq n(P\Psi-\norm{s_{o}}^{2})$ holds. However, our approach does not require that the sup-norm of the function $\Psi$ is bounded, we only need a finite expectation.
\end{remark}

\begin{proof}
Since \eqref{Cond.marg} holds, it comes from Theorem \ref{theo:FundamentalResult} that, for all $\Delta\geq 2$, the estimator (\ref{def:Estimator}) satisfies, with probability larger than $1-\delta$,
\begin{align}\label{eq.IntDensityMest}
\norm{\ERM_{\Ks}-\st}^{2}\leq &\norm{s_{o}-\st}^{2}+\paren{1+\frac{8}{\Delta}}\inf_{K}\set{\norm{\ERM_{K}-\st}^{2}}\\
\notag&+L_{1}^{2}\Delta\paren{P\Psi-\norm{s_{o}}^{2}}\frac{V}n\enspace.
\end{align}
Moreover, as $\E\paren{\norm{\ERM_{K}-s_{o}}^{2}}=\frac{P\Psi-\norm{s_{o}}^{2}}{|B_{K}|}$, by regularity of the partition, we deduce that
\[
\E\paren{\norm{\ERM_{K}-s_{o}}^{2}}\leq 2\paren{P\Psi-\norm{s_{o}}^{2}}\frac{V}n.
\]
Hence, from \eqref{eq.IntDensityMest} and Lemma \ref{lem:withmoments}, with probability larger than $1-2\delta$,
\[
\norm{\ERM_{\Ks}-\st}^{2}\leq \norm{s_{o}-\st}^{2}+\paren{2\sqrt{e}+\frac{16\sqrt{e}}{\Delta}+L_{1}^{2}\Delta}\paren{P\Psi-\norm{s_{o}}^{2}}\frac{V}n\enspace .
\]
We conclude the proof, choosing $\Delta=4e^{1/4}/L_{1}$.
\end{proof}

\subsubsection{Density estimation with K\"ullback loss}
We observe $X_{1:n}$ with density $\st$ with respect to the Lebesgue measure $\mu$ on $[0,1]$. We denote by $\Lcal$ the space of positive functions $t$ such that $P|\ln t|<\infty$. We have
\[
\st=\arg\min_{t\in \Lcal}\set{-P\ln(t)},\;\mbox{hence}\;\gamma(t)=-\ln t.
\]
Let $S$ be a space of histograms on $[0,1]$, {\it i.e.} there exists a partition $(\Il)_{\lL}$ of $[0,1]$, where, for all $\lL$, $\mu(\Il)\neq 0$, such that all functions $t$ in $S$ can be written $t=\sum_{\lL}\al\un_{\Il}$. We denote by
\[
s_{o}=\arg\min_{t\in S}\set{-P\ln t},\;\mbox{and}\; \forall K=1,\ldots,V,\;\tilde{s}_{K}=\arg\min_{t\in S}-P_{B_{K}}\ln t.
\]
It comes from Jensen's inequality that 
\[
s_{o}=\sum_{\lL}\frac{P\un_{\Il}}{\mu(\Il)}\un_{\Il},\;\mbox{and}\; \forall K=1,\ldots,V,\;\tilde{s}_{K}=\sum_{\lL}\frac{P_{B_{K}}\un_{\Il}}{\mu(\Il)}\un_{\Il}.
\]
Therefore, for all $K=1,\ldots,V$, 
\[
\frac{s_{o}}{\tilde{s}_{K}}=\sum_{\lL}\frac{P\un_{\Il}}{P_{B_{K}}\un_{\Il}}\un_{\Il}.
\]
For all $K=1,\ldots,V$, the estimator $\tilde{s}_{K}$ has a finite K\"ullback risk on the event 
\[
\Omega_{K}=\set{\forall \lL\;\mbox{such that}\;P\un_{\Il}>0,\;P_{B_{K}}\un_{\Il}>0}.
\]
In order to avoid this problem, we choose $x>0$ and define, for all $K=1,\ldots,V$, the estimator
\begin{equation}\label{def:estimKull}
\ERM_{K}=\frac{\tilde{s}_{K}+x\un_{[0,1]}}{1+x}.
\end{equation}
This way, $\ERM_{K}$ is always non-null and the K\"ullback risk of $\ERM_{K}$ is always finite. 
Finally, we denote by $C_{reg}(S)$ a constant such that
\begin{align*}\label{cond.reg.histo}
&\tag{\ensuremath{\mathbf{CR}}} \min_{\lL\telque P\Il\neq 0}P\Il\geq C_{reg}^{-1}(S).
\end{align*}
The following result ensures that \eqref{Cond.marg} holds.
\begin{proposition}\label{prop:C1DenistyKull}
Let $X_{1:n}$ be random variables with density $\st$ with respect to the Lebesgue measure $\mu$ on $[0,1]$. Let $S$ be a space of histograms on $[0,1]$. Let $V\leq n$ be an integer and let $\B$ be a regular partition of $\set{1,\ldots,n}$. For all $K=1,\ldots,V$, let $\ERM_{K}$ be the estimator defined in (\ref{def:estimKull}). Then, \eqref{Cond.marg} holds with $\sigma_{0}=0$, $N=1$, $\alpha_{1}=1/2$ and
\(
\sigma_{1}^{2}= 2+3\ln(1+x^{-1}) \enspace .
\)
\end{proposition} 
\noindent
Proposition \ref{prop:C1DenistyKull} is proved in Section \ref{Proof.Prop.C1DensKull}.

\begin{remark}
As for the $L^{2}$-loss, no assumptions on the observations are required to check \eqref{Cond.marg}.
\end{remark}

\begin{proposition}\label{prop:DenistyKull}
Let $X_{1:n}$ with density $\st$ with respect to the Lebesgue measure $\mu$ on $[0,1]$. Let $S$ be a space of histograms on $[0,1]$ with finite dimension $D$ and let $C_{reg}(S)$ be a constant satisfying condition \eqref{cond.reg.histo}. Let $\delta>0$ such that $\PESup{\ln(\delta^{-2})}\leq n/2$. Let $\B$ be a regular partition of $\set{1,\ldots,n}$, with $V=\PESup{\ln(\delta^{-2})}\vee 8$. For all $K=1,\ldots,V$, let $\ERM_{K}$ be the estimator defined in (\ref{def:estimKull}), with $x=n^{-1}$ and let $\ERM_{\Ks}$ be the associated estimator defined in (\ref{def:Estimator}). Then, for some absolute constant $L_{6}$,
\[
\P\paren{\ell(\ERM_{\Ks},\st)> \ell(s_{o},\st)+L_{6}\frac{D\ln(n)V}n\paren{1+\frac{C_{reg}(S)}{n}}}\leq 2\delta \enspace .
\]
\end{proposition}

\noindent
Proposition \ref{prop:DenistyKull} is proved in Section \ref{proof.prop.densityKull}.

\subsubsection{Regression with $L^{2}$-risk}
Let $(X_{i},Y_{i})_{i=1,...,n}$ be independent copies of a pair of random variables $(X,Y)$ satisfying the following equation
\[
Y=\st(X)+\sigma(X)\epsilon,\quad \mbox{with}\quad \E(\epsilon|X)=0,\;\E(\epsilon^{2}|X)=1,\;\sigma,\st\in L^{2}(P_{X}).
\]
Let $t$ in $L^{2}(P_{X})$, we have
\[
\E\paren{(Y-t(X))^{2}}=\E\paren{(Y-t(X))^{2}\sachant X}=\E\paren{(\st(X)-t(X))^{2}+\sigma^{2}(X)}.
\]
Hence, $\st=\arg\min_{t\in L^{2}(P_{X})}P\gamma(t)$, with $\gamma(t)(X,Y)=(Y-t(X))^{2}$. Moreover, we have
\[
\ell(t,\st)=\E\paren{(\st(X)-t(X))^{2}+\sigma^{2}(X)}-\E\paren{\sigma^{2}(X)}=\E\paren{(\st(X)-t(X))^{2}}
\]
Let $S$ be a linear space of functions and let $s_{o}$ be the orthogonal projection of $\st$ onto $S$ in $L^{2}(P_{X})$. Let $B=\set{t\in S,\norm{t}_{L^{2}(P_{X})}\leq 1}$, $\Psi=\sup_{t\in B}t^{2}$. We assume the following moments assumptions. There exist finite $D$ and $M_{\Psi}$ such that
\begin{equation}\label{cond.moment.reg}
\tag{\ensuremath{\mathbf{CM}}} \E\paren{(Y-s_{o}(X))^{2}\Psi(X)}\leq D,\qquad \E(\Psi^{2}(X))\leq M_{\Psi}.
\end{equation}
The following Proposition ensures that Condition \ref{cond.moment.reg} implies Condition \eqref{Cond.marg}.
\begin{proposition}\label{prop:C1RegL2}
Let $(X,Y), ((X_{i},Y_{i}))_{i=1,\ldots n}$ be i.i.d. pairs of random variables such that, 
\[
Y=\st(X)+\sigma(X)\epsilon,\; \mbox{where}\; \E(\epsilon|X)=0,\;\E(\epsilon^{2}|X)=1,\;\st,\sigma\in L^{2}(P_{X}) \enspace.
\]
Let $S$ be a linear space of functions measurable with respect to $L^{2}(P_{X})$. Let $s_{o}$ be the orthogonal projection of $\st$ onto $S$ and assume the moment condition \eqref{cond.moment.reg}. Let $V\leq n$ be an integer and let $\B$ be a regular partition of $\set{1,\ldots,n}$. For all $K=1,\ldots,V$, let $\ERM_{K}=\arg\min_{t\in S}P_{B_{K}}\gamma(t)$. Then, \eqref{Cond.marg} holds, with $\sigma_{0}^{2}=2M_{\Psi}$, $N=1$, $\alpha_{1}=1/2$, $\sigma_{1}^{2}=8D.$
\end{proposition}

\noindent
Proposition \ref{prop:C1RegL2} is proved in Section \ref{proof.prop.C1L2Reg}.
We can now derive the following consequence of Theorem \ref{theo:FundamentalResult}.
\begin{proposition}\label{prop:RegL2}
Let $(X,Y), ((X_{i},Y_{i}))_{i=1,\ldots n}$ be i.i.d. pairs of random variables such that, 
\[
Y=\st(X)+\sigma(X)\epsilon,\; \mbox{where}\; \E(\epsilon|X)=0,\;\E(\epsilon^{2}|X)=1,\;\st,\sigma\in L^{2}(P_{X}).
\]
Let $S$ be a linear space of functions measurable with respect to $L^{2}(P_{X})$. Let $s_{o}$ be the orthogonal projection on $\st$ onto $S$ and assume the moment condition \eqref{cond.moment.reg}. Let $\delta>0$ such that $\PESup{\ln(\delta^{-2})}\leq n/2$. Let $\B$ be a regular partition of $\set{1,\ldots,n}$, with $V=\PESup{\ln(\delta^{-2})}\vee 8$. For all $K=1,\ldots,V$, let $\ERM_{K}=\arg\min_{t\in S}P_{B_{K}}\gamma(t)$ and let $\ERM_{\Ks}$ be the associated estimator defined in (\ref{def:Estimator}). If $96eM_{\Psi}V\leq n$, then, for $L_{7}=384+128\sqrt{2}eL_{1}$,
\[
\P\paren{\ell(\ERM_{\Ks},\st)\leq \ell(s_{o},\st)+L_{7}\frac{DV}n}\geq 1-3\delta.
\]
\end{proposition}

\noindent
Proposition \ref{prop:RegL2} is proved in Section \ref{proof.prop.RegL2}.

\section{Extension to mixing data}\label{section.mixing}
An interesting feature of our approach is that it can easily be adapted to deal with mixing data, using coupling methods as for example in \cite{BCV01, CM02, Le09, Le10}. Let us recall the definition of $\beta$-mixing and $\phi$-mixing coefficients, due respectively to \cite{RV59} and \cite{Ib62}. Let $(\Omega,\mathcal{A},\P)$ be a probability space and let $\Xcal$ and $\Ycal$ be two $\sigma$-algebras included in $\mathcal{A}$. We define
\begin{align*}
\beta(\Xcal,\Ycal)&=\frac12\sup\set{\sum_{i=1}^{I}\sum_{j=1}^{J}\absj{\P\set{A_{i}\cap B_{j}}-\P\set{A_{i}}\P\set{B_{j}}}}\enspace,\\
\phi(\Xcal,\Ycal)&=\sup_{A\in\Xcal,\;\P\set{A}>0}\sup_{B\in\Ycal}\P\set{B|A}-\P\set{B}\enspace .
\end{align*}
The first $\sup$ is taken among the finite partitions $(A_{i})_{i=1,\ldots,I}$ and $(B_{j})_{j=1,\ldots,J}$ of $\Omega$ such that, for all $i=1,\ldots,I$, $A_{i}\in \Xcal$ and for all $j=1,\ldots,J$, $B_{j}\in \Ycal$.\\
For all stationary sequences of random variables $(X_n)_{n\in\Z}$ defined on $(\Omega,\mathcal{A},\P)$, let
\begin{equation*}
\beta_k=\beta(\sigma(X_i,  i\leq 0),\sigma(X_i,  i\geq k)),\qquad \phi_{k}=\phi(\sigma(X_i,  i\leq 0),\sigma(X_i,  i\geq k))\enspace .
\end{equation*}
The process $(X_n)_{n\in\Z}$ is said to be $\beta$-mixing when $\beta_k\rightarrow 0$ as $k\rightarrow \infty$, it is said to be $\phi$-mixing when $\phi_k\rightarrow 0$ as $k\rightarrow \infty$. It is easily seen, see for example inequality (1.11) in \cite{Br05} that $\beta(\Xcal,\Ycal)\leq \phi(\Xcal,\Ycal)$ so that $(X_n)_{n\in\Z}$ is $\phi$-mixing implies $(X_n)_{n\in\Z}$ is $\beta$-mixing.

\subsection{Basic concentration inequality}
In this section, we assume that $n=2Vq$. For all $K=1,\ldots,2V$, let $B_{K}=\set{(K-1)q+1,\ldots, Kq}$, $\B_{mix}=(B_{1},\ldots,B_{2V})$. For every $f:\R\mapsto\R$, let 
\(
\olP_{mix}f=\olP_{\B_{mix}}f.
\)
Our first result is the extension of Proposition \ref{pro.conc.rob} to mixing processes.
\begin{proposition}
Let $X_{1:n}$ be a stationary, real-valued, $\beta$-mixing process and assume that
\[
C_{\beta}^{2}=2\sum_{l\geq 0}(l+1)\beta_l<\infty. 
\]
There exists an event $\Omega_{coup}$ satisfying $\P\set{\Omega_{coup}}\geq 1-2V\beta_{q}$ such that, for all $f$ such that $\norm{f}_{4,P}\egaldef \paren{Pf^{4}}^{1/4}<\infty$, for all $V\geq \ln(2\delta^{-1})$, we have, for $L_{8}=4\sqrt{6e}$,
\[
\P\set{\olP_{mix}f-Pf>L_{8}\sqrt{C_{\beta}}\norm{f}_{4,P}\sqrt{\frac{V}{n}}\cap \Omega_{coup}}\leq \delta\enspace .
\]
If moreover, $X_{1:n}$ is $\phi$-mixing, with $\sum_{i=0}^{n}\phi_{q}\leq \Phi^{2}$, then, for all $f$ such that $Pf^{2}<\infty$, for all $V\geq \ln(2\delta^{-1})$,
\[
\P\set{\olP_{mix}f-Pf>L_{8}\Phi\sqrt{\var_{P}f}\sqrt{\frac{V}{n}}\cap \Omega_{coup}}\leq \delta\enspace .
\]
\end{proposition}
\begin{remark}
This result allows to extend the propositions of Section \ref{sect.RobEstMean} and \ref{sect.FirstApp}. We see that we only have to modify slightly the estimation procedure, choosing $\olP_{mix}$ instead of $\olP_{\B}$ to deal with these processes and that the price to pay is to work under higher moments conditions. Under these stronger conditions, all the results remain valid.
\end{remark}
\begin{proof}
For all $x>0$ and all $V$, if the following bounds are satisfied 
\begin{align*}
 \med&\set{P_{B_{K}}f,\;K=1,3,\ldots,2V-1}-Pf\leq x\\
 \med&\set{P_{B_{K}}f,\;K=2,4,\ldots,2V}-Pf\leq x\enspace .
\end{align*}
Then, there is at least $V$ values of $P_{B_{K}}f$ smaller than $Pf+x$, so that $\olP_{mix}f-Pf\leq x$. Hence
\begin{align*}
\P\set{\olP_{mix}f-Pf>x}\leq &\P\set{\med\set{P_{B_{K}}f,\;K=1,3,\ldots,2V-1}-Pf>x}\\
&+\P\set{\med\set{P_{B_{K}}f,\;K=2,4,\ldots,2V}-Pf>x}
\end{align*}
The proof is then a consequence of Lemma \ref{lem.ProofConcMix} below.
\end{proof}

\begin{lemma}\label{lem.ProofConcMix}
Let $X_{1:n}$ be a stationary, real-valued, $\beta$-mixing process and assume that
\[
C_{\beta}^{2}=2\sum_{l\geq 0}(l+1)\beta_l<\infty. 
\]
For all $a\in\set{0,1}$, there exists an event $\Omega^{a}_{coup}$ satisfying $\P\set{\Omega^{a}_{coup}}\geq 1-V\beta_{q}$ such that, for all $f$ such that $\norm{f}_{4}\egaldef \paren{Pf^{4}}^{1/4}<\infty$, for all $V\geq \ln(2\delta^{-1})$, we have, $\P\paren{(\Omega^{a}_{\beta})^{c}\cap \Omega^{a}_{coup}}\leq \delta/2$, where $\Omega_{\beta}^{a}$ is the set
\[
\med\set{P_{B_{K}}f,\;K=1+a,3+a,\ldots,2V-1+a}-Pf\leq L_{8}\sqrt{C_{\beta}}\norm{f}_{4}\sqrt{\frac{V}{n}}\enspace .
\]
If moreover, $X_{1:n}$ is $\phi$-mixing, with $\sum_{i=0}^{n}\phi_{q}\leq \Phi^{2}$, then, for all $f$ such that $Pf^{2}<\infty$, for all $V\geq \ln(2\delta^{-1})$,  we have, $\P\paren{(\Omega^{a}_{\phi})^{c}\cap \Omega^{a}_{coup}}\leq \delta/2$, where $\Omega_{\phi}^{a}$ is the set
\[
\med\set{P_{B_{K}}f,\;K=1+a,3+a,\ldots,2V-1+a}-Pf\leq L_{8}\Phi\sqrt{\var_{P}f}\sqrt{\frac{V}{n}}\enspace .
\]
\end{lemma}
\noindent
Lemma \ref{lem.ProofConcMix} is proved in Section \ref{sect.proof.conc.Mixing}.

\subsection{Construction of Estimators}
The purpose of this section is to adapt the results of Section \ref{sect.Mest} to our mixing setting. Let $V\geq 8$ and assume that $n$ can be divided by $2V$. Let us write $q=n/(2V)$ and, for all $K=1,\ldots,2V$, $B_{K}=\set{(K-1)q+1,\ldots, Kq}$. Let us then denote by $\Ical=\cup_{K=1}^{V}B_{2K-1}\subset \set{1,\ldots,n}$. We define now, for all $K\neq K^{\prime}=1,\ldots,V$, 
\[
\olP^{mix}_{K,K^{\prime}}=\med_{J\in \Ical-\set{B_{2K-1},B_{2K^{\prime}-1}}}P_{B_{J}}t.
\]
Let us assume, as in Section \ref{sect.Mest}, that, for all $K=1,\ldots,V$, $\ERM_{K}=F(X_{B_{2K-1}})$. Our final estimator is defined now by $\ERM^{mix}=\ERM_{\Ks^{mix}},$ where
\begin{equation}\label{def:EstimatorMixing}
\Ks^{mix}=\arg\min_{K=1,...,V}\max_{K^{\prime}=1,\ldots,V}\set{\olP^{mix}_{K,K^{\prime}}\paren{\gamma(\ERM_{K})-\gamma(\ERM_{K^{\prime}})}}.
\end{equation}
Our result is the following.
\begin{theorem}\label{theo:ResultMixing}
Let $X_{1:n}$ be $\phi$-mixing random variables with $\sum_{q=1}^{n}\phi_{q}\leq \Phi^{2}$. Let $\delta>0$ such that $\PESup{\ln(2\delta^{-2})}\leq n/2$ and let $V=\PESup{\ln(2\delta^{-2})}\vee 16$. Let $(\ERM_{K})_{K=1,\ldots,V}$ be a sequence of estimators such that $\ERM_{K}=F_{K}(X_{2K-1})$ and assume that \eqref{Cond.marg} holds. Let $\ERM^{mix}$ be the associated estimator defined in \eqref{def:EstimatorMixing}.
Let $C_{0}=L_{8}\Phi,\;\mbox{and}\;\forall i=1,\ldots,N,\;C_{i}=(1-\alpha_{i})\paren{C_{0}(\alpha_{i})^{\alpha_{i}}}^{1/(1-\alpha_{i})}$.
For all $\Delta>1$, let 
\[
\nu_{n}(\Delta)=C_{0}\sigma_{0}\sqrt{\frac{V}{n}}+\frac N{\Delta},\; R_{n}(\Delta)=\sum_{i=1}^{N}C_{i}\paren{\Delta^{\alpha_{i}}\sigma_{i}\sqrt{\frac{V}{n}}}^{\frac1{1-\alpha_{i}}}.
\]
For all $\Delta>1$, we have, with probability larger than $1-\delta-\epsilon-V\beta_{q}$,
\[
\paren{1-\nu_{n}(\Delta)}\ell(\ERM^{mix},s_{o})\leq  \paren{1+3\nu_{n}(\Delta)}\inf_{K=1,\ldots,V}\set{\ell(\ERM_{2K-1},s_{o})}+R_{n}(\Delta)\enspace.
\]
In particular, for all $\Delta$, $n$ such that $\nu_{n}(\Delta)\leq 1/2$, we have, with probability larger than $1-\epsilon-\delta-V\beta_{q}$,
\[
\ell(\ERM^{mix},\st)\leq  \ell(s_{o},\st)+\paren{1+8\nu_{n}(\Delta)}\inf_{K=1,\ldots,V}\set{\ell(\ERM_{2K-1},s_{o})}+2R_{n}(\Delta).
\]
\end{theorem}
\begin{remark}
The only price to pay to work with these data is therefore the $V\beta_{q}$ in the control of the probability and a small improvements of the constants. When $\delta, \epsilon=O( n^{-2})$ and $\beta_{q}\leq Cq^{-(1+\theta)}$ for $\theta>1$, we obtain that this probability is $O(n^{-2})$ in both cases, the price is negligible. 
\end{remark}
\subsection{Application to density estimation}
Proposition \ref{prop.CondMarg.Denisty} ensures that Condition \eqref{Cond.marg} holds for all stationary processes and all estimators.
In order to extend Propostion \ref{prop:DensityL2} to mixing data, we only have to extend the inequality on $\E\paren{\norm{\ERM_{K}-s_{o}}^{2}}$. The result is the following.
\begin{proposition}
Let $X_{1:n}$ be a $\phi$-mixing process such that $\sum_{q=1}^{n}\phi_{q}\leq \Phi$, with common marginal density $\st$. Assume that $\st\in L^{2}(\mu)$. Let $\delta>0$ such that $\PESup{\ln(\delta^{-2})}\leq n/2$ and let $V=\PESup{\ln(\delta^{-2})}\vee 16$. Let $B_{1},\ldots,B_{2V}$ be a regular partition of $\set{1,\ldots,n}$. For all $K$, let $\ERM_{K}=\arg\min_{t\in S}P_{B_{K}}\gamma(t)$ and let $\ERM^{mix}$ be the associated estimator defined in (\ref{def:EstimatorMixing}). We have
\[
\P\paren{\norm{\ERM^{mix}-\st}^{2}>\norm{s_{o}-\st}^{2}+C\Phi^{2}P\Psi\frac{V}n}\leq \delta+V\beta_{q}.
\]
\end{proposition}
\begin{proof}
As, on $\Omega_{good}$ the estimators are equal to those built using the independent data $(A_{K})_{K=1,\ldots,V}$, we only have to prove, thanks to Theorem \ref{theo:ResultMixing} and Lemma \ref{lem:withmoments} that
\[
\E\paren{\norm{\ERM_{K}-s_{o}}^{2}}\leq C\Phi^{2}P\Psi\frac{V}n
\]
to obtain the result. 
%We adapt the proof of Lemma 5.1 in \cite{Le10}. We recall the following inequality for $\beta$-mixing processes (see for example \cite{CM02} inequalities (6.2) and (6.3)). Let $(X_n)_{n\in\Z}$ be $\beta$-mixing data. There exists a function $\nu$ satisfying, for all $p,q$ in $\N$,
%%
%\begin{equation}\label{eq:CovBetaBasic}
%\nu=\sum_{l\geq 0}\nu_l,\;{\rm with}\;P\nu_q\leq \beta_q\;{\rm and}\;P(\nu^p)\leq p\sum_{l\geq 0}(l+1)^{p-1}\beta_l,
%\end{equation}
%%
%such that, for all $t$ in $L^2(\mu)$,
%%
%\begin{equation}\label{cov:Vi97}
%\absj{\var\paren{\sum_{i=1}^qt(X_i)}}\leq 4qP(\nu t^2).
%\end{equation}
We have, from inequality (\ref{cov:Vi97}),
\begin{align*}
\E\paren{\norm{\ERM_{K}-s_{o}}^{2}}&=\sum_{\lL}\var\paren{\frac{1}{|B_{K}|}\sum_{i\in B_{K}}\psil(X_{i})}\leq\frac4{|B_{K}|}P\paren{\nu\Psi}\leq \frac{4\Phi^{2}P\Psi}{|B_{K}|}.
\end{align*}
\end{proof}

\subsection*{Acknowledgements}
We want to thank Antonio Galves for many discussions and fruitful advices during the redaction of the paper.
 
ML was supported by FAPESP grant 2009/09494-0. This work is part of USP project ``Mathematics, computation, language and the brain''.

\appendix

\section{Proofs of the basic results}
\subsection{Proof of Proposition \ref{pro.conc.rob}}\label{sect.proof.Prop.conc.rob}
We have 
\[
\olP_{\B} f-Pf=\med\set{P_{B_{K}}f-Pf,\;K=1,\ldots,V}\enspace.
\]
Denote, for all $x>0$, by $N_{x}=\card\set{K=1,\ldots,V, \telque P_{B_{K}}f-Pf>x}$. For all $V\leq n$, we have
\begin{align}
\notag\P&\set{\olP_{\B} f-Pf>x}\leq \P\set{N_{x}\geq \frac V2}\enspace .
\end{align}

%\notag&=\sum_{k=V/2}^{V}\P\set{\card\set{K=1,\ldots,V, \telque P_{K}f-Pf>x}=k}\\
%\label{eq.intConcIneq}&=\sum_{k=V/2}^{V}\binom{V}{k}\paren{\max_{K=1,\ldots,V}\P\set{P_{K}f-Pf>x}}^{k}\enspace.
%\end{align}
Now let $r>\sqrt{2}$ to be chosen later. It comes from Tchebychev's inequality that
\begin{equation}\label{Cond.Bloc}
\forall K=1,\ldots,V,\qquad \P\set{P_{B_{K}}f-Pf> \sqrt{\frac{\var_{P}f}{r}}\sqrt{\frac1{\card\set{B_{K}}}}}\leq r \enspace.
\end{equation}
Let us also introduce a random variable $B$ with binomial distribution $B(V,r)$. It comes from Lemma \ref{lem.coupl} that
\begin{align*}
\P\set{\olP_{\B} f-Pf>\sqrt{\frac2r}\sqrt{\var_{P}f}\sqrt{\frac Vn}}\leq \P\set{B\geq \frac V2} \enspace.
\end{align*}
Now, it comes from equation (2.8) in Chapter 2 of \cite{Ma07} that
\[
\P\set{B\geq \frac V2}\leq e^{-\frac V2\ln\paren{\frac1{4r(1-r)}}}.
\]
We choose $r=(1-\sqrt{1-e^{-2}})/2$, so that $4r(1-r)=e^{-2}$. For all $V\geq \ln(\delta^{-1})$, we have
\begin{equation}\label{interm.concIn2}
\P\set{\olP_{\B} f-Pf>\sqrt{\frac2r}\sqrt{\var_{P}f}\sqrt{\frac Vn}}\leq\delta.
\end{equation}
Finally, we have $r\geq 1/(12e)$.
%
%}
\subsection{Proof of Proposition \ref{pro.interm.Lasso}}\label{sect.Proof.Prop.Interm.Lasso}
Let us first remark that, for all $\theta\in \Theta$, we have
\begin{align*}
\crit(\theta,\B)&=\norm{\sthet}^{2}-2\sum_{\lL}\thetal \olP_{\B}\psil+2\sum_{\lL}\omegal\absj{\thetal}\\
&=\norm{\sthet}^{2}-2\sum_{\lL}\thetal P\psil+2\sum_{\lL}\thetal (P-\olP_{\B})\psil+2\sum_{\lL}\omegal\absj{\thetal}\\
&=\norm{\sthet-\st}^{2}-\norm{\st}^{2}+2\sum_{\lL}\thetal (P-\olP_{\B})\psil+2\sum_{\lL}\omegal\absj{\thetal}\enspace.
\end{align*}
By definition of $\thetah$, we have therefore, for all $\theta\in \Theta$,
\begin{align*}
\norm{s_{\thetah}-\st}^{2}\leq &\norm{\sthet-\st}^{2}+2\sum_{\lL}(\thetal-\thetahl) (P-\olP_{\B})\psil\\
&+2\sum_{\lL}\omegal\absj{\thetal}-2\sum_{\lL}\omegal\absj{\thetahl}\enspace.
\end{align*}
Let $\Omega_{good}$ be the event 
\[
\forall \lL,\qquad \absj{\olP \psil-P\psil}\leq  L_{2}\sqrt{\olP_{\B}\psil^{2}}\sqrt{\frac{V}{n}}\enspace .
\] 
Since $V\geq \ln(4M\delta^{-1})$ and \eqref{hyp.Dic} holds, it comes from Corollary \ref{coro.ConcMean} for the dictionary $\Dcal$ and the uniform probability measure on $\Lambda$ that $\P\set{\Omega_{good}}\geq 1-\delta$. Moreover, on $\Omega_{good}$
\[
\forall \lL,\qquad 2\absj{(P-\olP_{\B})\psil}\leq \omegal \enspace.
\]
On $\Omega_{good}$, for all $\theta\in\Theta$, by the triangular inequality, we have then
\begin{align*}
&\norm{s_{\thetah}-\st}^{2}+\sum_{\lL}\omegal\absj{\thetal-\thetahl}\\
&\leq \norm{s_{\theta}-\st}^{2}+2\sum_{\lL}\omegal\absj{\thetal-\thetahl} +2\sum_{\lL}\omegal\absj{\thetal}-2\sum_{\lL}\omegal\absj{\thetahl}\\
&\leq \norm{s_{\theta}-\st}^{2}+2\sum_{\lambda\in J(\theta)}\omegal\absj{\thetal-\thetahl}+2\sum_{\lambda\in J(\theta)}\omegal\absj{\thetal}-2\sum_{\lambda\in J(\theta)}\omegal\absj{\thetahl}\\
&\leq  \norm{s_{\theta}-\st}^{2}+4\sum_{\lambda\in J(\theta)}\omegal\absj{\thetal-\thetahl}\enspace.
\end{align*}
\subsection{Proof of Theorem \ref{theo.Estim.Select.Classic}}\label{Proof.theo.Estim.Select.Classic}
Let us first remark that, for all $m_{o}\in\M$,
\begin{align*}
\crit_{\alpha}(\theta)+\norm{\st}^{2}+2(P_{n}-P)s_{m_{o}}&\\
=\min_{\mM_{\theta}}\left\{\norm{\ERMP_{\theta,m}-\st}^{2}-2(P_{n}-\right.&\left.P)(\ERMP_{\theta,m}-s_{m}-s_{m_{o}}+s_{m})\right.\\
&\left.+\alpha\norm{\ERMP_{\theta,m}-\ERM_{\theta}}^{2}+\pen(m)\right\}\enspace .
\end{align*}
For all $\mM$, we denote by $(\psil)_{\lL_{m}}$ an orthonormal basis of $S_{m}$ and by $\betahl^{\theta}=\psh{\ERMP_{\theta,m}}{\psil}$. Using Cauchy-Schwarz inequality and the inequality $2ab\leq \epsilon a^{2}+\epsilon^{-1}b^{2}$, for all $\theta\in\Theta$, for all $\mM_{\theta}$, we obtain
\begin{align*}
2\absj{(P_{n}-P)(\ERMP_{\theta,m}-s_{m})}&=2\absj{\sum_{\lL_{m}}(\betahl^{\theta}-P\psil)\croch{(P_{n}-P)\psil}}\\
&\leq2\sqrt{\sum_{\lL_{m}}(\betahl^{\theta}-P\psil)^{2}}\sqrt{\sum_{\lL_{m}}\croch{(P_{n}-P)\psil}^{2}}\\
&\leq\epsilon\sum_{\lL_{m}}(\betahl^{\theta}-P\psil)^{2}+\frac1{\epsilon}\sum_{\lL_{m}}\croch{(P_{n}-P)\psil}^{2}\\
&=\epsilon \norm{\ERMP_{\theta,m}-s_{m}}^{2}+\frac{\sum_{\lL_{m}}\croch{(P_{n}-P)\psil}^{2}}{\epsilon}.
\end{align*}
Let now $\Omega_{good}$ be the event
\begin{align*}
\forall \mM,&\qquad \absj{(P_{n}-P)(s_{m_{o}}-s_{m})}\leq \varepsilon_{n}^{\prime}\paren{\norm{s_{m}-\st}^{2}+\frac{P\Psi_{m}}n}\\
&\sum_{\lL_{m}}\croch{(P_{n}-P)\psil}^{2}\leq (1+L_{0}\nu)\frac{P\Psi_{m}}n+\frac{L_{o}\norm{s}}{\nu}r_{m}(\delta)+\frac{L_{o}}{\nu^{3}}r^{2}_{m}(\delta)\enspace .
\end{align*}
It comes from Lemma \ref{lem.conc.density.Classic} in the appendix that, for all $\nu>0$ and all $\delta>\delta_{o}$, $\P\set{\Omega_{good}}\geq 1-\delta$. Moreover, on $\Omega_{good}$, for all $\epsilon>0$, we have
\begin{align*}
\left|2(P_{n}-P)\right.&\left.(\ERMP_{\theta,m}-s_{m}-s_{m_{o}}+s_{m})\right|\leq \varepsilon_{n}^{\prime}\paren{\norm{s_{m}-\st}^{2}+\frac{P\Psi_{m}}n}\\
&+\epsilon \norm{\ERMP_{\theta,m}-s_{m}}^{2}+\frac{(1+L_{0}\nu)\frac{P\Psi_{m}}n+\frac{L_{o}\norm{s}}{\nu}r_{m}(\delta)+\frac{L_{o}}{\nu^{3}}r^{2}_{m}(\delta)}{\epsilon}\enspace .
\end{align*}
We choose $\epsilon= 1/2$ and $n\geq n_{o}$. Since
\[
\pen(m)\geq \paren{\frac52+2L_{0}\nu}\frac{P\Psi_{m}}n+\frac{2L_{o}\norm{s}}{\nu}r_{m}(\delta)+\frac{2L_{o}}{\nu^{3}}r^{2}_{m}(\delta)\enspace ,
\]
on $\Omega_{good}$, using the triangular inequality, we have then
\begin{align*}
\crit_{\alpha}(\theta)&+\norm{\st}^{2}+2(P_{n}-P)s_{m_{o}}\\
&\leq \min_{\mM_{\theta}}\set{\frac32\norm{\ERMP_{\theta,m}-\st}^{2}+\alpha\norm{\ERMP_{\theta,m}-\ERM_{\theta}}^{2}+2\pen(m)}\\
&\leq 3\norm{\ERM_{\theta}-\st}^{2}+\min_{\mM_{\theta}}\set{(3+\alpha)\norm{\ERMP_{\theta,m}-\ERM_{\theta}}^{2}+2\pen(m)}\enspace ,\\
\crit_{\alpha}(\theta)&+\norm{\st}^{2}+2(P_{n}-P)s_{m_{o}}\\
&\geq \min_{\mM_{\theta}}\set{\frac12\norm{\ERMP_{\theta,m}-\st}^{2}+\alpha\norm{\ERMP_{\theta,m}-\ERM_{\theta}}^{2}}\\
&\geq \paren{\frac14\wedge\frac{\alpha}2}\norm{\ERM_{\theta}-\st}^{2}\enspace .
\end{align*}
By definition of $\thetah$, it follows that, on $\Omega_{good}$, $\forall \theta\in\Theta$
\begin{align*}
\paren{\frac14\wedge\frac{\alpha}2}&\norm{\ERM_{\thetah}-\st}^{2}\\
&\leq 3\norm{\ERM_{\theta}-\st}^{2}+\min_{\mM_{\theta}}\set{(3+\alpha)\norm{\ERMP_{\theta,m}-\ERM_{\theta}}^{2}+2\pen(m)}\enspace .
\end{align*}
\subsection{Proof of Theorem \ref{theo.EstSelRob}}\label{sect.proof.theo.Estim.select.Rob}
The proof follows essentially the same steps as the one of Theorem \ref{theo.Estim.Select.Classic}. Let $m_{o}$ be a minimizer of $2\epsilon\norm{\st-s_{m}}^{2}+4e(\epsilon n)^{-1}\sum_{\lL_{m}}\var_{P}\psil V_{\lambda}$. We have
\begin{align*}
\crit_{\alpha}(\theta)&+\norm{\st}^{2}+2\sum_{\lL_{m_{o}}}P\psil(\olP_{\B_{\lambda}}-P)\psilp\\
&=\min_{\mM_{\theta}}\left\{\norm{\ERMP_{\theta,m}-\st}^{2}-2\sum_{\lL_{m}}(\betahl^{\theta}-P\psil)(\olP_{\B_{\lambda}}\psil-P\psil)\right.\\
&\hspace{2cm}\left.+2\sum_{\lL_{m}\cup\Lambda_{m_{o}}}\al(\olP_{\B_{\lambda}}-P)\psil+\alpha\norm{\ERMP_{\theta,m}-\ERM_{\theta}}^{2}+\pen(m)\right\}\enspace ,
\end{align*}
where, for all $\lL_{m}\cup\Lambda_{m_{o}}$, $\al=P\psil\paren{\un_{\lL_{m_{o}}}-\un_{\lL_{m}}}$, so that $\sum_{\lL_{m}\cup\Lambda_{m_{o}}}\al^{2}=\norm{s_{m}-s_{m_{o}}}^{2}$.
Using Cauchy-Schwarz inequality and the inequality $2ab\leq \epsilon a^{2}+\epsilon^{-1}b^{2}$, for all $\theta\in\Theta$, for all $\mM_{\theta}$, we obtain
\begin{align}
\notag2\sum_{\lL_{m}}\absj{\betahl^{\theta}-P\psil}&\absj{\olP_{\B_{\lambda}}\psil-P\psil}\\
&\leq2\sqrt{\sum_{\lL_{m}}(\betahl^{\theta}-P\psil)^{2}}\sqrt{\sum_{\lL_{m}}\croch{(\olP_{\B_{\lambda}}\psil-P\psil)}^{2}}\\
\notag&\leq\eta\sum_{\lL_{m}}(\betahl^{\theta}-P\psil)^{2}+\frac1{\eta}\sum_{\lL_{m}}\croch{(\olP_{\B_{\lambda}}\psil-P\psil)}^{2}\\
\label{eq.control.Main.RobustEstSelect}&=\eta \norm{\ERMP_{\theta,m}-s_{m}}^{2}+\frac1{\eta}\sum_{\lL_{m}}\croch{(\olP_{\B_{\lambda}}\psil-P\psil)}^{2}\enspace .
\end{align}
\begin{align}
\notag2\sum_{\lL_{m}\cup\Lambda_{m_{o}}}&\absj{\al(\olP_{\B_{\lambda}}-P)\psil}\\
\notag&\leq2\sqrt{\sum_{\lL_{m}\cup\Lambda_{m_{o}}}\al^{2}}\sqrt{\sum_{\lL_{m}\cup\Lambda_{m_{o}}}\croch{(\olP_{\B_{\lambda}}-P)\psil}^{2}}\\
\label{eq.control.Remainder.RobustEstSelect}&\leq \epsilon \norm{s_{m}-s_{m_{o}}}^{2}+\frac1\epsilon\sum_{\lL_{m}\cup\Lambda_{m_{o}}}\croch{(\olP_{\B_{\lambda}}-P)\psil}^{2}\\
\notag&\leq 2\epsilon\norm{\st-s_{m_{o}}}^{2}+2\epsilon\norm{s_{m}-\st}^{2}+\frac1{\epsilon}\paren{\sum_{\lL_{m}\cup\Lambda_{m_{o}}}\croch{(\olP_{\B_{\lambda}}-P)\psil}^{2}}.
\end{align}
Let now $\Omega_{good}$ be the event
\begin{align*}
\forall \lL_{M},&\qquad \absj{\olP_{\B_{\lambda}}\psil-P\psil}\leq L_{1}\sqrt{\var_{P}(\psil)}\sqrt{\frac{V_{\lambda}}n}\enspace .
\end{align*}
Using a union bound in Proposition \ref{pro.conc.rob}, since $V_{\lambda}\geq \ln(2(\pi(\lambda)\delta)^{-1})$, we have, $\P\set{\Omega_{good}}\geq 1-\delta/2$. Moreover, on $\Omega_{good}$, from \eqref{eq.control.Main.RobustEstSelect}. for all $\eta>0$, we have

\begin{align*}
2\sum_{\lL_{m}}\absj{\betahl^{\theta}-P\psil}\absj{\olP_{\B_{\lambda}}\psil-P\psil}\leq \eta \norm{\ERMP_{\theta,m}-s_{m}}^{2}+\frac{L_{1}^{2}}{\eta n}\sum_{\lL_{m}}\var_{P}(\psil)V_{\lambda}.
\end{align*}
From \eqref{eq.control.Remainder.RobustEstSelect}, for all $\epsilon>0$, on $\Omega_{good}$,
\begin{align*}
2&\sum_{\lL_{m}\cup\Lambda_{m_{o}}}\absj{\al(\olP_{\B_{\lambda}}-P)\psil}\\
&\leq 2\epsilon\norm{\st-s_{m_{o}}}^{2}+2\epsilon\norm{s_{m}-\st}^{2}+\frac{L_{1}^{2}}{n\epsilon}\paren{\sum_{\lL_{m}\cup\Lambda_{m_{o}}}\var_{P}(\psil)V_{\lambda}}\\
&\leq 2\epsilon\norm{\st-s_{m_{o}}}^{2}+2\epsilon\norm{s_{m}-\st}^{2}\\
&\quad+\frac{L_{1}^{2}}{n\epsilon}\paren{\sum_{\lL_{m}}\var_{P}(\psil)V_{\lambda}+\sum_{\lL_{m_{o}}}\var_{P}(\psil)V_{\lambda}}\\
&\leq 4\epsilon\norm{s_{m}-\st}^{2}+\frac{L_{2}^{2}}{n\epsilon}\sum_{\lL_{m}}\var_{P}(\psil)V_{\lambda}\enspace .
\end{align*}
The last inequality is due to the definition of $m_{o}$. We choose $\eta=4\epsilon$, on $\Omega_{good}$, using the triangular inequality, we deduced that, for $L_{4}=L_{2}^{2}+L_{1}^{2}/4=9L_{1}^{2}/4$, $\crit_{\alpha}(\theta)+\norm{\st}^{2}+2\sum_{\lL_{m_{o}}}P\psil(\olP_{\B_{\lambda}}\psil-P\psil)$ is smaller than the infimum over $\mM_{\theta}$
\begin{align*}
&(1+4\epsilon)\norm{\ERMP_{\theta,m}-\st}^{2}+\frac{L_{4}}{n\epsilon}\sum_{\lL_{m}}\var_{P}(\psil)V_{\lambda}+\alpha\norm{\ERMP_{\theta,m}-\ERM_{\theta}}^{2}+\pen(m)\\
&\leq4\norm{\ERM_{\theta}-\st}^{2}+\paren{4+\alpha}\norm{\ERMP_{\theta,m}-\ERM_{\theta}}^{2}+\frac{L_{4}}{n\epsilon}\sum_{\lL_{m}}\var_{P}(\psil)V_{\lambda}+\pen(m)\enspace .
\end{align*}
On the other hand, $\crit_{\alpha}(\theta)+\norm{\st}^{2}+2\sum_{\lL_{m_{o}}}P\psil(\olP_{\B_{\lambda}}\psil-P\psil)$ is larger than the infimum over $\mM_{\theta}$ of
\begin{align*}
&(1-4\epsilon)\norm{\ERMP_{\theta,m}-\st}^{2}-\frac{L_{4}}{n\epsilon}\sum_{\lL_{m}}\var_{P}(\psil)V_{\lambda}+\alpha\norm{\ERMP_{\theta,m}-\ERM_{\theta}}^{2}+\pen(m)\\
&\geq \frac12\paren{(1-4\epsilon)\wedge \alpha}\norm{\ERM_{\theta}-\st}^{2}-\frac{L_{4}}{n\epsilon}\sum_{\lL_{m}}\var_{P}(\psil)V_{\lambda}+\pen(m)\enspace .
\end{align*}
Since 
\[
\pen(m)\geq \frac{L_{4}}{n\epsilon}\sum_{\lL_{m}}\var_{P}(\psil)V_{\lambda} \enspace ,
\]
by definition of $\thetah$, we deduce that, on $\Omega_{good}$, for all $\theta\in\Theta$,
\begin{align*}
 \frac12&\paren{(1-4\epsilon)\wedge \alpha}\norm{\ERM_{\thetah}-\st}^{2}\\
 &\leq 4\norm{\ERM_{\theta}-\st}^{2}+\min_{\mM_{\theta}}\left\{\paren{4+\alpha}\norm{\ERMP_{\theta,m}-\ERM_{\theta}}^{2}+2\pen(m)\right\}\enspace .
\end{align*}
\section{Proofs for $M$-estimation}
\subsection{Proof of Theorem \ref{theo:FundamentalResult}}\label{Sect.proof.fundresult}
Let $\Omega_{(C)}$ be the event 
\[
\max_{K}\norm{\var\croch{\paren{\gamma(\ERM_{K})-\gamma(s_{o})}(X)\sachant X_{B_{K}}}}_{\infty}\leq \sigma_{0}^{2}\ell(\ERM_{K},s_{o})^{2}+\sum_{i=1}^{N}\sigma_{i}^{2}\ell(\ERM_{K},s_{o})^{2\alpha_{i}}. 
\]
We apply Proposition \ref{pro.conc.rob} to $f=\paren{\gamma(\ERM_{K})-\gamma(s_{o})}\un_{\Omega_{(C)}}$ and to $-f$, conditionally to the random variables $X_{B_{K}}$ and to the partition $\B=(B_{J})_{J\neq K,K'}$ of $\set{1,\ldots,n}/(B_{K}\cup B_{K'})$, with cardinality $n-|B_{K}|-|B_{K'}|\geq n(1-2/V-2/n)>n/2$ since $V\geq8$. We have 
\[
\var(f(X)|X_{B_{K}})\leq \sigma_{0}^{2}\ell(\ERM_{K},s_{o})^{2}+\sum_{i=1}^{N}\sigma_{i}^{2}\ell(\ERM_{K},s_{o})^{2\alpha_{i}}\enspace .
\] 
Since $V\geq \ln(\delta^{-2})$, Proposition \ref{pro.conc.rob} gives that, with probability larger than $1-2\delta^{2}$, conditionally to $X_{B_{K}}$, the following event holds
\begin{align*}
\Omega_{(C)}\cap&\absj{(\olP_{K,K'}-P)\paren{\gamma(\ERM_{K})-\gamma(s_{o})}}\\
&\qquad>L_{1}\sqrt{\sigma_{0}^{2}\ell(\ERM_{K},s_{o})^{2}+\sum_{i=1}^{N}\sigma_{i}^{2}\ell(\ERM_{K},s_{o})^{2\alpha_{i}}}\sqrt{\frac{V}n} \\
\supset\Omega_{(C)}\cap&\absj{(\olP_{K,K'}-P)\paren{\gamma(\ERM_{K})-\gamma(s_{o})}}\\
&\qquad>L_{1}\paren{\sigma_{0}\ell(\ERM_{K},s_{o})+\sum_{i=1}^{N}\sigma_{i}\ell(\ERM_{K},s_{o})^{\alpha_{i}}}\sqrt{\frac{V}n}\enspace .
\end{align*}
As the bound on the probability does not depend on $X_{B_{K}}$, the same bound holds unconditionally. We use repeatedly the classical inequality $a^{\alpha}b^{1-\alpha}\leq \alpha a+(1-\alpha)b$. Let $r_{n}=\sqrt{n^{-1}V}$, let $C_{0}=L_{1}$ and, for all $i=1,\ldots,N$ let $C_{i}=(1-\alpha_{i})\paren{L_{1}(\alpha_{i})^{\alpha_{i}}}^{1/(1-\alpha_{i})}$, we obtain that the following event has probability smaller than $2\delta^{2}$,
\begin{align*}
\Omega_{(C)}\cap&\absj{(\olP_{K,K'}-P)\paren{\gamma(\ERM_{K})-\gamma(s_{o})}}\\
&\qquad \leq \paren{C_{0}\sigma_{0}r_{n}+\frac N{\Delta}}\ell(\ERM_{K},s_{o})+\sum_{i=1}^{N}C_{i}\paren{\Delta^{\alpha_{i}}\sigma_{i}r_{n}}^{\frac1{1-\alpha_{i}}}.
\end{align*}
Using a union bound, we get that the probability that the following event has probability larger than $1-V(V-1)\delta^{2}-\epsilon$, $\forall K,K'=1,\ldots,V,$
\begin{align*}
&\left|(\olP_{K,K'}-P)\paren{\gamma(\ERM_{K})-\gamma(\ERM_{K'})}\right|\\ 
&\quad\leq\paren{C_{0}\sigma_{0}r_{n}+\frac N{\Delta}}\paren{\ell(\ERM_{K},s_{o})+\ell(\ERM_{K'},s_{o})}+2\sum_{i=1}^{N}C_{i}\paren{\Delta^{\alpha_{i}}\sigma_{i}r_{n}}^{\frac1{1-\alpha_{i}}}.
\end{align*}
Given the value of $V$, we obtain that 
\begin{align*}
V(V-1)\delta^{2}&= \paren{\PESup{\ln(\delta^{-2})}\paren{\PESup{\ln(\delta^{-2})}-1}}\delta^{2}\leq \paren{\paren{\ln(\delta^{-2})+1}\ln(\delta^{-2})}\delta^{2}\\
&\leq \delta\paren{\sup_{\delta\in(0,1)}\set{\delta\paren{\paren{\ln(\delta^{-2})+1}\ln(\delta^{-2})}}}\leq \delta.
\end{align*}
Hereafter, we denote by 
\[
\nu_{n}(\Delta)=C_{0}\sigma_{0}r_{n}+\frac N{\Delta},\quad R_{n}(\Delta)=2\sum_{i=1}^{N}C_{i}\paren{\Delta^{\alpha_{i}}\sigma_{i}r_{n}}^{\frac1{1-\alpha_{i}}}
\]
and by $\Omega_{good}(\delta)$ the event,
\begin{align*}
\bigcap_{K,K'=1,\ldots,V}&\left\{\absj{(\olP_{K,K'}-P)\paren{\gamma(\ERM_{K})-\gamma(\ERM_{K'})}}\right.\\
&\left.\qquad\leq \nu_{n}(\Delta)\paren{\ell(\ERM_{K},s_{o})+\ell(\ERM_{K'},s_{o})}+R_{n}(\Delta)\right\}.
\end{align*}
Let $K_{o}$ be the index such that $P\gamma(\ERM_{K})$ is minimal. For all $n$, $\Delta$ sufficiently large so that $\nu_{n}(\Delta)<1$, on $\Omega_{good}$, we have
\begin{align*}
&\ell(\ERM_{\Ks},s_{o})=\paren{P\paren{\gamma(\ERM_{\Ks})-\gamma(\ERM_{K_{o}})}}+\ell(\ERM_{K_{o}},s_{o})\\
&= \croch{P\paren{\gamma(\ERM_{\Ks})-\gamma(\ERM_{K_{o}})}-\nu_{n}(\Delta)\paren{\ell(\ERM_{\Ks},s_{o})+\ell(\ERM_{K_{o}},s_{o})}}\\
&+\nu_{n}(\Delta)\paren{\ell(\ERM_{\Ks},s_{o})}+\paren{1+\nu_{n}(\Delta)}\ell(\ERM_{K_{o}},s_{o})\\
&=\sup_{K}\croch{P\paren{\gamma(\ERM_{\Ks})-\gamma(\ERM_{K})}-\nu_{n}(\Delta)\paren{\ell(\ERM_{\Ks},s_{o})+\ell(\ERM_{K},s_{o})}}\\
&+\nu_{n}(\Delta)\paren{\ell(\ERM_{\Ks},s_{o})}+\paren{1+\nu_{n}(\Delta)}\ell(\ERM_{K_{o}},s_{o})\\
&\leq\sup_{K}\croch{\olP_{\Ks,K}\paren{\gamma(\ERM_{\Ks})-\gamma(\ERM_{K})}}\\
&+\nu_{n}(\Delta)\paren{\ell(\ERM_{\Ks},s_{o})}+\paren{1+\nu_{n}(\Delta)}\ell(\ERM_{K_{o}},s_{o})+R_{n}(\Delta).
\end{align*}
By definition of $\Ks$, we have
\begin{align*}
&\paren{1-\nu_{n}(\Delta)}\ell(\ERM_{\Ks},s_{o})\\
&\quad\leq \sup_{K}\croch{\olP_{K_{o},K}\paren{\gamma(\ERM_{\Ks})-\gamma(\ERM_{K})}}+\paren{1+\nu_{n}(\Delta)}\ell(\ERM_{K_{o}},s_{o})+R_{n}(\Delta)\\
&\quad \leq\sup_{K}\croch{P\paren{\gamma(\ERM_{K_{o}})-\gamma(\ERM_{K})}+\nu_{n}(\Delta)\paren{\ell(\ERM_{K_{o}},s_{o})+\ell(\ERM_{K},s_{o})}}\\
&\qquad+\paren{1+\nu_{n}(\Delta)}\ell(\ERM_{K_{o}},s_{o})+2R_{n}(\Delta)\\
&\quad=\sup_{K}\croch{\paren{1+\nu_{n}(\Delta)}\ell(\ERM_{K_{o}},s_{o})-\paren{1-\nu_{n}(\Delta)}\ell(\ERM_{K},s_{o})}\\
&\qquad+\paren{1+\nu_{n}(\Delta)}\ell(\ERM_{K_{o}},s_{o})+2R_{n}(\Delta)\\
&\quad=\paren{1+3\nu_{n}(\Delta)}\ell(\ERM_{K_{o}},s_{o})+2R_{n}(\Delta)\enspace .
\end{align*}
This concludes the proof of Theorem \ref{theo:FundamentalResult}.

\subsection{Proof of Proposition \ref{prop:C1DenistyKull}}\label{Proof.Prop.C1DensKull}
We have
\[
\frac{s_{o}}{\ERM_{K}}=\sum_{\lL}\frac{(1+x)P\un_{\Il}}{P_{B_{K}}\un_{\Il}+x}\un_{\Il}.
\]
We deduce that
\[
\ln\paren{\frac{s_{o}}{\ERM_{K}}}\leq\ln\paren{\frac{(1+x)}{x}}=\ln\paren{x^{-1}+1}.
\]
We have, by Cauchy-Schwarz inequality,
\begin{align}
\ell(\ERM_{K},s_{o})&=\int s_{o}\ln\paren{\frac{s_{o}}{\ERM_{K}}}=\int_{s_{o}>\ERM_{K}}s_{o}\ln\paren{\frac{s_{o}}{\ERM_{K}}}-\int_{s_{o}<\ERM_{K}}s_{o}\ln\paren{\frac{\ERM_{K}}{s_{o}}}\nonumber\\
&\geq \int_{s_{o}>\ERM_{K}}s_{o}\ln\paren{\frac{s_{o}}{\ERM_{K}}}-\int_{s_{o}<\ERM_{K}}s_{o}\paren{\ln\paren{\frac{\ERM_{K}}{s_{o}}}}^{2}\label{eq:int:Kull1}
\end{align}
Let us now recall the following lemma, see for example \cite{Ma07} Lemma 7.24.
\begin{lemma}\label{lem:BarSheu}
For all probability measures $P$, $Q$ with $P<<Q$,
\[
\frac12\int(dP\wedge dQ)\paren{\ln\paren{\frac{dP}{dQ}}}^{2}\leq \int dP\ln\paren{\frac{dP}{dQ}}
\]
\end{lemma}

\noindent
Using Lemma \ref{lem:BarSheu}, we get
\[
\int_{s_{o}<\ERM_{K}}s_{o}\paren{\ln\paren{\frac{\ERM_{K}}{s_{o}}}}^{2}\leq \int s_{o}\wedge \ERM_{K}\paren{\ln\paren{\frac{s_{o}}{\ERM_{K}}}}^{2}\leq 2\int s_{o}\ln\paren{\frac{s_{o}}{\ERM_{K}}}.
\]
Plugging this inequality in (\ref{eq:int:Kull1}), we obtain
\[
\int_{s_{o}>\ERM_{K}}s_{o}\ln\paren{\frac{s_{o}}{\ERM_{K}}}\leq 3\int s_{o}\ln\paren{\frac{s_{o}}{\ERM_{K}}}.
\]
Finally, we obtain
\begin{align*}
\var&\paren{\gamma(\ERM_{K}-\gamma(s_{o})\sachant X_{B_{K}}}\leq\int s_{o}\paren{\ln\paren{\frac{s_{o}}{\ERM_{K}}}}^{2}\\
&= \int_{s_{o}<\ERM_{K}} s_{o}\paren{\ln\paren{\frac{s_{o}}{\ERM_{K}}}}^{2}+\int_{s_{o}>\ERM_{K}} s_{o}\paren{\ln\paren{\frac{s_{o}}{\ERM_{K}}}}^{2}\\
&\leq 2\int s_{o}\ln\paren{\frac{s_{o}}{\ERM_{K}}}+\ln(x^{-1}+1)\int_{s_{o}>\ERM_{K}} s_{o}\ln\paren{\frac{s_{o}}{\ERM_{K}}}\\
&\leq (2+3\ln(x^{-1}+1))\ell(\ERM_{K},s_{o}).
\end{align*}
\eqref{Cond.marg} holds with $\sigma_{0}=0$, $N=1$, $\alpha_{1}=1/2$, 
\[
\sigma_{1}^{2}= 2+3\ln(1+x^{-1}) \enspace .
\]

\subsection{Proof of Proposition \ref{prop:DenistyKull}}\label{proof.prop.densityKull}
Condition \eqref{Cond.marg} holds from Proposition \ref{prop:C1DenistyKull}. Hence, Theorem \ref{theo:FundamentalResult} ensures that, for all $\Delta\geq 2$, the estimator (\ref{def:Estimator}) satisfies, with probability larger than $1-\delta$,
\[
\ell(\ERM_{\Ks},\st)\leq  \ell(s_{o},\st)+\paren{1+\frac{8}{\Delta}}\inf_{K}\set{\ell(\ERM_{K},\st)}+L_{1}^{2}\Delta(2+3\ln (1+x^{-1}))\frac{V}n\enspace.
\]
Let us now fix some $K$ in $1,\ldots,V$ and let us denote, for all $\lL$ by $\al=P(\Il)$, $\ahl=\frac{P_{B_{K}}(\Il)+x}{1+x}$. We have
\begin{align*}
\E\paren{\ell(\ERM_{K},s_{o})}&=\E\paren{\int \paren{\sum_{\lL}\frac{\al}{\mu(\Il)}\ln\paren{\frac{\al}{\ahl}}}d\mu}\\
&= \sum_{\lL,\al\neq0}\al\E\paren{\ln\paren{\frac{\al}{\ahl}}}\enspace .
\end{align*}
Now, we have
\[
\ln\paren{\frac{\al}{\ahl}}=-\ln\paren{1-\frac{\al-\ahl}{\al}},\;\mbox{with}\;\frac{\al-\ahl}{\al}=1-\frac{\ahl}{\al}\leq 1-\frac{x}{(1+x)} \enspace .
\]
We use the following inequalities, for all $u\leq 1-\Gamma^{-1}$
\[
\frac{-\ln(1-u)-u}{u^{2}}\leq \frac{\Gamma^{2}\ln(\Gamma)}{(\Gamma-1)^{2}}+\frac{\Gamma}{\Gamma-1}\enspace .
\]
we deduce that
\[
\ln\paren{\frac{\al}{\ahl}}\leq\frac{\al-\ahl}{\al}+\paren{(1+x)^{2}\ln\paren{\frac{1+x}x}+1+x}\paren{\frac{\al-\ahl}{\al}}^{2}\enspace .
\]
Moreover, we have,
\[
0<x\leq 1,\;\mbox{and}\; \E\paren{\al-\ahl}\leq 0,\;\E\paren{\paren{\ahl-\al}^{2}}\leq \frac{\al}{|B_{K}|}+x^{2}\enspace .
\]
We deduce that, for all $\lL$ such that $\al\neq0$,
\begin{align*}
\al\E\paren{\ln\paren{\frac{\al}{\ahl}}}&=\al\E\paren{-\ln\paren{1-\frac{\al-\ahl}{\al}}}\\
&\leq \E\paren{\al-\ahl}+\frac{4\ln(2x^{-1})+2}{\al}\E\paren{\paren{\ahl-\al}^{2}}\\
&\leq \paren{4\ln(2x^{-1})+2}\paren{\frac1{|B_{K}|}+\frac{x^{2}}{\al}}\\
&\leq \paren{4\ln(2x^{-1})+2}\paren{\frac1{|B_{K}|}+x^{2}C_{reg}(S)}.
\end{align*}
We deduce that
\[
\E\paren{\ell(\ERM_{K},s_{o})}\leq \paren{4\ln(2x^{-1})+2}D\paren{\frac{V}{n}+x^{2}C_{reg}(S)}.
\]
Hence, applying Lemma \ref{lem:withmoments} with $\alpha=1$ and the previous bound in expectation, we obtain
\[
\P\set{\inf_{K=1,\ldots,V}\ell(\ERM_{K},s_{o})>\sqrt{e}\paren{4\ln(2x^{-1})+2}D\paren{\frac{V}{n}+x^{2}C_{reg}(S)}}\leq \delta.
\]
This concludes the proof, choosing $\Delta=2$ and $x=n^{-1}$.

\subsection{Proof of Proposition \ref{prop:C1RegL2}}\label{proof.prop.C1L2Reg}
We have
\begin{align*}
\var&\paren{\gamma(\ERM_{K})-\gamma(s_{o})\sachant X_{B_{K}}}\\
&=\var\paren{\paren{\ERM_{K}(X)-s_{o}(X)}^{2}+2\paren{s_{o}(X)-\ERM_{K}(X)}\paren{Y-s_{o}(X)}\sachant X_{B_{K}}}\\
&\leq 2\var\paren{\paren{\ERM_{K}(X)-s_{o}(X)}^{2}\sachant X_{B_{K}}}\\
&\quad+8\var\paren{\paren{s_{o}(X)-\ERM_{K}(X)}\paren{Y-s_{o}(X)}\sachant X_{B_{K}}}\\
&\leq 2\E\paren{\paren{\ERM_{K}(X)-s_{o}(X)}^{4}\sachant X_{B_{K}}}\\
&\quad+8\E\paren{\paren{s_{o}(X)-\ERM_{K}(X)}^{2}\paren{Y-s_{o}(X)}^{2}\sachant X_{B_{K}}}.
\end{align*}
Let us now consider an orthonormal basis (in $L^{2}(P_{X})$) $(\psil)_{\lL}$ of $S$ and let us write 
\[
s_{o}=\sum_{\lL}\al\psil,\;\ERM_{K}=\sum_{\lL}\ahl(K)\psil,\;\mbox{hence}\;P(s_{o}-\ERM_{K})^{2}=\sum_{\lL}(\al-\ahl(K))^{2}.
\]
From Cauchy-Schwarz inequality $\Psi=\sum_{\lL}\psil^{2}$, hence, from Cauchy-Schwarz inequality, we have
\begin{align*}
\paren{\ERM_{K}(X)-s_{o}(X)}^{2}&=\paren{\sum_{\lL}(\al-\ahl(K))\psil}^{2}\\
&\leq \paren{\sum_{\lL}(\al-\ahl(K))^{2}}\paren{\sum_{\lL}\psil^{2}}=\ell(\ERM_{K},s_{o})\Psi.
\end{align*}
Thus, we have
\[
\E\paren{\paren{\ERM_{K}(X)-s_{o}(X)}^{4}\sachant X_{B_{K}}}\leq (\ell(\ERM_{K},s_{o}))^{2}\E\paren{\Psi^{2}(X)}=M_{\Psi}(\ell(\ERM_{K},s_{o}))^{2}.
\]
Moreover, from Cauchy-Schwarz inequality,
\[
\E\paren{\paren{s_{o}(X)-\ERM_{K}(X)}^{2}\paren{Y-s_{o}(X)}^{2}\sachant X_{B_{K}}}\leq \ell(\ERM_{K},s_{o})D.
\]
We deduce that \eqref{Cond.marg} holds, with $\sigma_{0}^{2}=2M_{\Psi}$, $N=1$, $\alpha_{1}=1/2$, $\sigma_{1}^{2}=8D.$

\subsection{Proof of Proposition \ref{prop:RegL2}}\label{proof.prop.RegL2}
Theorem \ref{theo:FundamentalResult} and $144M_{\Psi}V\leq n$ ensure that, for all $\Delta\geq 4$, we have $\nu_{n}(\Delta)=6\sqrt{\frac{M_{\Psi}V}{n}}+\frac{1}{\Delta}\leq \frac12$, hence, the estimator (\ref{def:Estimator}) satisfies, with probability larger than $1-\delta$,
\begin{equation}\label{eq.Basic.Reg}
\ell(\ERM_{\Ks},\st)\leq \ell(s_{o},\st)+\paren{1+8\nu_{n}(\Delta)}\inf_{K}\set{\ell(\ERM_{K},\st)}+8L_{1}^{2}\Delta\frac{DV}n.
\end{equation}
In order to control $\inf_{K}\set{\ell(\ERM_{K},\st)}$, we use the following method, due to Saumard \cite{Sa11}. We fix $K$ in $1,\ldots V$ and, for all constants $C$, we denote by 
\[
\G_{C}=\set{t\in S,\;\ell(t,s_{o})\leq C},\qquad \G_{>C}=\set{t\in S,\;\ell(t,s_{o})>C}.
\]
We have, by definition of $\ERM_{K}$,
\begin{align*}
\ell&(\ERM_{K},s_{o})>C \Rightarrow \inf_{t\in \G_{C}}P_{B_{K}}(\gamma(t)-\gamma(s_{o}))\geq \inf_{t\in \G_{>C}}P_{B_{K}}(\gamma(t)-\gamma(s_{o}))\\
&\Rightarrow \sup_{t\in \G_{C}}P_{B_{K}}(\gamma(s_{o})-\gamma(t))\leq \sup_{t\in \G_{>C}}P_{B_{K}}(\gamma(s_{o})-\gamma(t))\\
&\Rightarrow \sup_{t\in \G_{C}}\set{(P_{B_{K}}-P)(\gamma(s_{o})-\gamma(t))-\ell(t,s_{o})}\\
&\qquad \leq \sup_{t\in \G_{>C}}\set{(P_{B_{K}}-P)(\gamma(s_{o})-\gamma(t))-\ell(t,s_{o})}\\
&\Rightarrow\sup_{t\in \G_{>C}}\set{(P_{B_{K}}-P)(\gamma(s_{o})-\gamma(t))-\ell(t,s_{o})}\geq 0.
\end{align*}
Let us now write
\begin{align*}
\gamma(s_{o})-\gamma(t)&=(Y-s_{o}(X))^{2}-(Y-t(X))^{2}\\
&=-2(Y-s_{o}(X))(s_{o}(X)-t(X))-(s_{o}(X)-t(X))^{2} \enspace .
\end{align*}
Given an orthonormal basis (in $L^{2}(P_{X})$) $(\psil)_{\lL}$ of $S$, we write
\[
s_{o}(X)-t(X)=\sum_{\lL}\al\psil.
\]
Hence, we have, for all $t$ in $S$, for all $\epsilon>0$, using Cauchy-Schwarz inequality,
\begin{align*}
(P_{B_{K}}&-P)(\gamma(s_{o})-\gamma(t))\\
&=-2\sum_{\lL}\al(P_{B_{K}}-P)\paren{(Y-s_{o})\psil}-\sum_{\llpL}\al\alp(P_{B_{K}}-P)(\psil\psilp)\\
&\leq 2\sqrt{\sum_{\lL}\al^{2}}\sqrt{\sum_{\lL}\paren{(P_{B_{K}}-P)\paren{(Y-s_{o})\psil}}^{2}}\\
&\quad+\sum_{\lL}\al^{2}\sqrt{\sum_{\llpL}\paren{(P_{B_{K}}-P)(\psil\psilp)}^{2}}\\
&\leq \frac1{\epsilon}\sum_{\lL}\paren{(P_{B_{K}}-P)\paren{(Y-s_{o})\psil}}^{2}\\
&\quad+\paren{\epsilon+\sqrt{\sum_{\llpL}\paren{(P_{B_{K}}-P)(\psil\psilp)}^{2}}}\ell(t,s_{0}).
\end{align*}
We have obtained that if $\ell(\ERM_{K},s_{o})>C$ then
\begin{align*}
\sup_{t\in \G_{>C}}\left\{\frac1{\epsilon}\right.&\left.\sum_{\lL}\paren{(P_{B_{K}}-P)\paren{(Y-s_{o})\psil}}^{2}\right.\\
&\left.-\paren{1-\epsilon-\sqrt{\sum_{\llpL}\paren{(P_{B_{K}}-P)(\psil\psilp)}^{2}}}\ell(t,s_{0})\right\}\geq 0.
\end{align*}
Let $\Omega_{K}=\set{\sum_{\llpL}\paren{(P_{B_{K}}-P)(\psil\psilp)}^{2}\leq 1/4}$ and take $\epsilon=1/4$, we deduce from the previous bound that, on $\Omega_{K}$,
\[
\ell(\ERM_{K},s_{o})>C\Rightarrow 16\sum_{\lL}\paren{(P_{B_{K}}-P)\paren{(Y-s_{o})\psil}}^{2}\geq C.
\]
This means that $\ell(\ERM_{K},s_{o})\un_{\Omega_{K}}\leq 16\sum_{\lL}\paren{(P_{B_{K}}-P)\paren{(Y-s_{o})\psil}}^{2}$, in particular,
\begin{align}
\notag\E\paren{\ell(\ERM_{K},s_{o})\un_{\Omega_{K}}}&\leq 16\frac{\sum_{\lL}\var{\paren{(Y-s_{o})\psil}}}{|B_{K}|}\leq 16\frac{\E\paren{\paren{Y-s_{o}}^{2}\Psi(X)}}{|B_{K}|}\\
\label{eq.Cond.Exp}&=16\frac{D}{|B_{K}|}\leq 32\frac{D V}n\enspace .
\end{align}
We have
\[
\E\paren{\sum_{\llpL}\paren{(P_{B_{K}}-P)(\psil\psilp)}^{2}}\leq \frac{\E(\Psi^{2})}{|B_{K}|}\leq 2\frac{M_{\Psi}V}{n}\enspace .
\]
Hence, if $r=(1-\sqrt{1-e^{-2}})/2$, by Markov property, 
\begin{equation*}
\P\paren{\sum_{\llpL}\paren{(P_{B_{K}}-P)(\psil\psilp)}^{2}>\frac{2M_{\Psi}V}{r n}}\leq r \enspace.
\end{equation*}
As $r\geq 1/(12e)$ and $24e\frac{M_{\Psi}V}n< \frac14$, we deduce
\begin{equation}\label{eq.CondProbaOmegaK}
\P\paren{\Omega_{K}^{c}}\leq r \enspace.
\end{equation}
From Lemma \ref{lem.for.Regression} and equations \eqref{eq.Cond.Exp}, \eqref{eq.CondProbaOmegaK}, we deduce that
\[
\P\paren{\inf_{K=1,\ldots,V}\ell(\ERM_{K},s_{o})>128e^{2}\frac{D V}n}\leq 2e^{-V}\enspace .
\]
Together with \eqref{eq.Basic.Reg}, this yields, for $\Delta=\sqrt{2}8eL_{1}^{-1}$,
\[
\P\set{\ell(\ERM_{\Ks},\st)>\paren{384+128\sqrt{2}eL_{1}}\frac{DV}n}\leq \delta+2\delta^{2}\enspace .
\]

\section{Proofs in the mixing-case}
\subsection{Proof of Proposition \ref{lem.ProofConcMix}}\label{sect.proof.conc.Mixing}
Let $a\in\set{0,1}$ and let us recall the following lemma due to Viennet \cite{Vi97}.
\begin{lemma}\label{lem:Vi97} {\it (Lemma 5.1 in \cite{Vi97}) Assume that the process $(X_1,...,X_n)$ is $\beta$-mixing. There exist random variables $(A_{K})_{K=1,\ldots,V}$ such that:
\begin{enumerate}
\item $\forall K=1,\ldots,V$, $A_{K}$ has the same law as $(X_{i})_{i\in B_{2K-1+a}}$,
\item $\forall K=1,\ldots,V$, $A_{K}$ is independent of $(X_{i})_{i\in B_{1+a}},...,(X_{i})_{i\in B_{2(K-1)-1+a}}$, $A_{1},...,A_{K-1}$,
\item $\forall K=1,\ldots,V$, $\P((X_{i})_{i\in B_{2K-1+a}}\neq A_{K})\leq\beta_q$.
\end{enumerate}}
\end{lemma}
\noindent
Let us define the event 
\[
\Omega^{a}_{coup}=\set{\forall K=1,\ldots,V,\;(X_{i})_{i\in B_{2K-1+a}}= A_{K}}.
\]
It comes from Viennet's Lemma that $\P\set{\Omega^{a}_{coup}}\geq 1-V\beta_{q}$. For all $K=1,\ldots,V$, let $A_{K}=(Y_{i})_{i\in B_{K}}$ and for all measurable $t$, let 
\[
P_{A_{K}}t=\frac1{\card(B_{K})}\sum_{i\in B_{K}}t(Y_{i})\enspace.
\]
On $\Omega^{a}_{coup}$, we have $P_{B_{2K-1+a}}f=P_{A_{K}}f$, hence, 
\begin{align}\label{eq.conc.betamix}
&p(x)\egaldef\P\set{\med\set{P_{B_{K}}f,\;K=1,3,\ldots,2V-1+a}-Pf>x\cap\Omega^{a}_{coup}}\\
&\notag\leq \P\set{\card\set{K=1,3,\ldots,2V-1+a, \telque P_{B_{K}}f-Pf>x}\geq \frac V2\cap\Omega^{a}_{coup}}\\
&\notag\leq \P\set{\card\set{K=1,\ldots,V, \telque P_{A_{K}}f-Pf>x}\geq \frac V2}\enspace.
\end{align}
By Markov inequality, for all $r$, for all $K=1,\ldots,V$,
\begin{align*}
\P&\set{ P_{A_{K}}f-Pf>\sqrt{\frac{\var\paren{P_{B_{K}}f-Pf}}{r}}}\\
&\qquad\qquad=\P\set{ P_{B_{K}}f-Pf>\sqrt{\frac{\var\paren{P_{B_{K}}f-Pf}}{r}}}\leq r\enspace .
\end{align*}
From Lemma \ref{lem.coupl}, we deduce, denoting by $x_{r}=\sqrt{\frac{\var\paren{P_{B_{K}}f-Pf}}{r}}$ and by $B(V,r)$ a binomial random variable with parameters $V$ and $r$, that
\begin{equation}\label{eq.intermMix}
p(x_{r})\leq \P\set{B(V,r)\geq \frac V2}\leq e^{\frac V2\ln\paren{\frac1{4r(1-r)}}}\enspace .
\end{equation}
We choose $r=(1-\sqrt{1-e^{-2}})/2$ so that $\ln\paren{\frac1{4r(1-r)}}=2$. We recall now the following inequality for mixing processes (see for example \cite{CM02} inequalities (6.2) and (6.3) or Lemma 4.1 in \cite{Vi97}). 
\begin{lemma}\label{lem.cov.beta}
Let $(X_n)_{n\in\Z}$ be $\phi$-mixing data. There exists a function $\nu$ satisfying, for all $p,q$ in $\N$,
\begin{equation}\label{eq:CovBetaBasic}
\nu=\sum_{l\geq 0}\nu_l,\;{\rm with}\;P\nu_q\leq \beta_q,\;P(\nu^p)\leq p\sum_{l\geq 0}(l+1)^{p-1}\beta_l\;{\rm and}\;\norm{\nu_{q}}_{\infty}\leq \phi_{q}\enspace ,
\end{equation}
such that, for all $t$ in $L^2(\mu)$,
\begin{equation}\label{cov:Vi97}
\var\paren{\sum_{i=1}^qt(X_i)}\leq 4qP(\nu t^2) \enspace .
\end{equation}
\end{lemma}

\noindent
We apply Lemma \ref{lem.cov.beta} and we obtain, when $C_{\beta}<\infty$, $\forall K=1,3,\ldots,2V-1$
\[
\var\paren{P_{B_{K}}f-Pf}\leq \frac4{\card{B_{K}}}P(\nu f^{2})\leq \frac{8V}{n}\sqrt{P(\nu^{2})P( f^{4})} \enspace.
\]
Moreover, when $X_{1:n}$ is $\phi$-mixing, with $\sum_{q=1}^{n}\phi_{q}\leq \Phi^{2}$, we apply Lemma \ref{lem.cov.beta} to $f-Pf$ and we get $\forall K=1,3,\ldots,2V-1$,
\[
\var\paren{P_{B_{K}}f-Pf}\leq \frac4{\card{B_{K}}}P(\nu (f-Pf)^{2})\leq 8\Phi^{2}\frac{V}{n}\var_{P}f \enspace .
\]
Plugging these inequalities in \eqref{eq.intermMix} yields the result, since we have
\[
x_{r}\leq \sqrt{\frac{\var_{P}\paren{(P_{B_{K}}-P)f}}r}\leq \sqrt{8e\var_{P}\paren{(P_{B_{K}}-P)f}}\enspace .
\]
\subsection{Proof of Theorem \ref{theo:ResultMixing}}

Let us denote by $\Omega_{{\bf (C)}}$ the same event as in the proof of Theorem \ref{theo:FundamentalResult}.
We can apply Lemma \ref{lem.ProofConcMix} to the function $f=\paren{\gamma(\ERM_{K})-\gamma(s_{o})}\un_{\Omega_{{\bf (C)}}}$ conditionally to the random variables $X_{B_{2K-1}}$ and to the partition $(B_{2J-1})_{J\neq K,K'}$ of $\Ical/(B_{2K-1}\cup B_{2K'-1})$, with cardinality $n/2-|B_{2K-1}|-|B_{2K'-1}|\geq n(1/2-2/V-2/n)>n/4$ since $V\geq16$. Condition \eqref{Cond.marg} implies
\[
\var_{P}\paren{\paren{\gamma(\ERM_{K})-\gamma(s_{o})}\un_{\Omega_{{\bf (C)}}}\sachant X_{B_{2K-1}}}\leq \sigma_{0}^{2}\ell(\ERM_{K},s_{o})^{2}+\sum_{i=1}^{N}\sigma_{i}^{2}\ell(\ERM_{K},s_{o})^{2\alpha_{i}} \enspace .
\] 
We introduce
\[
C_{0}=8\sqrt{e}\Phi,\;\mbox{and}\;\forall i=1,\ldots,N,\;C_{i}=(1-\alpha_{i})\paren{C_{0}(\alpha_{i})^{\alpha_{i}}}^{1/(1-\alpha_{i})}.
\]
Since $V\geq \ln(2\delta^{-2})$, there exists a set $\Omega_{good}$ satisfying $\P\set{\Omega_{good}}\geq 1-V\beta_{q}$ such that, with probability larger than $1-\delta^{2}$, conditionally to $X_{B_{2K-1}}$, the following event holds
\begin{align*}
\Omega_{int}\cap&\absj{(\olP^{mix}_{K,K'}-P)\paren{\gamma(\ERM_{K})-\gamma(s_{o})}}\\
&\qquad \leq C_{0}\sqrt{\frac{V}n}\sqrt{\sigma_{0}^{2}\ell(\ERM_{K},s_{o})^{2}+\sum_{i=1}^{N}\sigma_{i}^{2}\ell(\ERM_{K},s_{o})^{2\alpha_{i}}} \\
\supset\Omega_{int}\cap&\absj{(\olP^{mix}_{K,K'}-P)\paren{\gamma(\ERM_{K})-\gamma(s_{o})}}\\
&\qquad \leq C_{0}\sqrt{\frac{V}n}\paren{\sigma_{0}\ell(\ERM_{K},s_{o})+\sum_{i=1}^{N}\sigma_{i}\ell(\ERM_{K},s_{o})^{\alpha_{i}}} .
\end{align*}
where $\Omega_{int}=\Omega_{good}\cap\Omega_{{\bf (C)}}$.
As the bound on the probability does not depend on $X_{B_{2K-1}}$, the same bound holds unconditionally. We use repeatedly the classical inequality $a^{\alpha}b^{1-\alpha}\leq \alpha a+(1-\alpha)b$, we obtain that
\begin{align*}
C_{0}\sqrt{\frac{V}n}&\paren{\sigma_{0}\ell(\ERM_{K},s_{o})+\sum_{i=1}^{N}\sigma_{i}\ell(\ERM_{K},s_{o})^{\alpha_{i}}}\\
&\leq \paren{C_{0}\sigma_{0}\sqrt{\frac{V}n}+\frac N{\Delta}}\ell(\ERM_{K},s_{o})+\sum_{i=1}^{N}C_{i}\paren{\Delta^{\alpha_{i}}\sigma_{i}\sqrt{\frac{V}n}}^{\frac1{1-\alpha_{i}}}.
\end{align*}
Using a union bound, we get that the following event holds with probability larger than $1-V(V-1)\delta^{2}-\epsilon-V\beta_{q}$, $\forall K,K'=1,\ldots,V$,
\begin{align*}
&\left|(\olP^{mix}_{K,K^{\prime}}-P)\paren{\gamma(\ERM_{K})-\gamma(\ERM_{K^{\prime}})}\right|\\
&\leq\paren{C_{0}\sigma_{0}\sqrt{\frac{V}n}+\frac N{\Delta}}\paren{\ell(\ERM_{K},s_{o})+\ell(\ERM_{K^{\prime}},s_{o})}+2\sum_{i=1}^{N}C_{i}\paren{\Delta^{\alpha_{i}}\sigma_{i}\sqrt{\frac{V}n}}^{\frac1{1-\alpha_{i}}}.
\end{align*}
We conclude as in the proof of Theorem \ref{theo:FundamentalResult}.
\section{Tools}
\subsection{A coupling result}
The aim of this section is to prove the coupling result used in the proof of Proposition \ref{pro.conc.rob}.
\begin{lemma}\label{lem.coupl}
Let $Y_{1:N}$ be independent random variables, let $x$ be a real number and let 
\[
A=\card\set{i=1,\ldots,N \telque Y_{i}>x}.
\]
Let $p\in(0,1]$ such that, for all $i=1,\ldots,N$, $p\geq \P\set{Y_{i}>x}$ and let $B$ be a random variable with a binomial law $B(N,p)$. There exists a coupling $\Ct=(\At,\Bt)$ such that $\At$ has the same distribution as $A$, $\Bt$ has the same distribution as $B$ and such that $\At\leq \Bt$. In particular, for all $y>0$, 
\begin{equation}\label{eq:comp.proba}
\P\set{A>y}\leq \P\set{B>y}\enspace .
\end{equation}
\end{lemma}
\begin{proof}
Let $U_{1:N}$ be i.i.d random variables with uniform distribution on $[0,1]$. Let us define, for all $i=1,\ldots,N$, 
\[
\At_{i}=\un_{U_{i}\leq \P\set{Y_{i}>x}},\qquad \Bt_{i}=\un_{U_{i}\leq p},\qquad \At=\sum_{i=1}^{N}\At_{i},\qquad\Bt=\sum_{i=1}^{N}\Bt_{i}\enspace .
\]
By construction, for all $i=1,\ldots,N$, $\At_{i}\leq \Bt_{i}$, hence $\At\leq \Bt$. Moreover, $\Bt$ is the sum of independent random variables with common Bernoulli distribution of parameter $p$, it has therefore the same distribution as $B$. We also have that $(\At_{i})_{i=1,\ldots,N}$ has the same distribution as $(\un_{Y_{i}>x})_{i=1,\ldots,N}$ since the marginals have the same distributions and the coordinates are independent. Since $A=\sum_{i=1}^{N}\un_{Y_{i}>x}$, $A$ and $\At$ have the same distribution.

In order to prove \eqref{eq:comp.proba}, we just say that $\At\leq \Bt$ implies that $\set{\At>y}\subset \set{\Bt>y}$, hence
\[
\P\set{A>y}=\P\set{\At>y}\leq \P\set{\Bt>y}=\P\set{B>y}\enspace .
\]
\end{proof}
\subsection{Concentration inequalities for Estimator Selection}
\begin{lemma}\label{lem.conc.density.Classic}
Let $(S_{m})_{\mM}$ be a collection of linear spaces of measurable functions and for all $\mM$, let $(\psil)_{\lL_{m}}$ be an orthonormal basis of $S_{m}$ and $\Psi_{m}=\sum_{\lL_{m}}\psil^{2}$. Let $\pi$ be a probability measure on $\M$. For all $\delta\in (0,1)$, we denote by 
\[
r_{m}(\delta)=\frac{\sqrt{\norm{\Psi_{m}}_{\infty}}\ln\paren{\frac{2}{\pi(m)\delta}}}{n}\enspace .
\]
Let $X_{1:n}$ be i.i.d, $\Xbf$-valued, random variables with common density $\st\in L^{2}(\mu)$. For all $\mM$, let $s_{m}$ be the orthogonal projection of $\st$ onto $S_{m}$ and assume that \eqref{Hyp.Sel.Est.Class} holds. Let us denote, for all $\delta\in(0,1)$, 
\[
\forall \mM,\;\eta_{\infty,2}(1\vee \norm{s})\sqrt{\norm{s_{m}-\st}^{2}+\frac{P\Psi_{m}}n}\ln\paren{\frac{4}{\delta\pi(m)}}\leq \varepsilon_{m}(\delta)\enspace.
\]
Then, with probability $1-\delta$, the following event holds. $\forall \mM$,
\begin{align*}
&\absj{(P_{n}-P)(s_{m_{o}}-s_{m})}\leq \frac{8\sqrt{2}}3\varepsilon_{m}(\delta)\paren{\norm{s_{m}-\st}^{2}+\frac{P\Psi_{m}}n}\\
& \sum_{\lL_{m}}\croch{(P_{n}-P)\psil}^{2}\leq (1+L_{0}\nu)\frac{P\Psi_{m}}n+\frac{L_{o}\norm{s}}{\nu}r_{m}(\delta)+\frac{L_{o}}{\nu^{3}}r^{2}_{m}(\delta)\enspace .
\end{align*}
\end{lemma}

\begin{proof}
Let $Z_{m}=\sum_{\lambda\in\Lambda_{m}}\croch{(P_{n}-P)\psil}^{2}$. We proved the following concentration inequality for $Z_{m}$ in \cite{Le09}. Let $v^{2}_{m}=\sup_{t\in\Boule(m)}P\croch{(t-Pt)^{2}}$, for some absolute $L_{0}\leq 16(\ln 2)^{-1}+8$, we have, for all $\delta>0$, $\nu>0$, with probability larger than $1-\delta/2$,
\[
Z_{m}-\frac{P\Psi_{m}}n\leq L_{0}\paren{\nu \frac{P\Psi_{m}}n+\frac{1}{\nu}\paren{\frac{v^{2}_{m}\ln(2/\delta)}n+\frac{\norm{\Psi_{m}}_{\infty}(\ln(2/\delta))^{2}}{\nu^{2}n^{2}}}}\enspace.
\]
We have
\[
v^{2}_{m}\leq \sqrt{\norm{\Psi_{m}}_{\infty}}\sup_{t\in\Boule(m)}P\croch{\absj{t}}\leq \sqrt{\norm{\Psi_{m}}_{\infty}}\norm{s}.
\]
Hence, if we denote by 
\[
r_{m}(\delta)=\frac{\sqrt{\norm{\Psi_{m}}_{\infty}}\ln\paren{\frac{2}{\pi(m)\delta}}}{n}\enspace ,
\]
we obtain,
\[
\P\set{\forall \mM,\quad Z_{m}>(1+L_{0}\nu)\frac{P\Psi_{m}}n+\frac{L_{o}\norm{s}}{\nu}r_{m}(\delta)+\frac{L_{o}}{\nu^{3}}r^{2}_{m}(\delta)}\leq \frac{\delta}2\enspace .
\]
Moreover, Bernstein's inequality yields, for all $\mM$, for all $m_{o}$, and for all $\delta\in (0,1)$, with probability larger than $1-\delta/2$,
\[
\absj{(P_{n}-P)(s_{m_{o}}-s_{m})}\leq \sqrt{\frac{2\var_{P}(s_{m}-s_{m_{o}})\ln(4/\delta)}n}+\frac{\norm{s_{m}-s_{m_{o}}}_{\infty}\ln(4/\delta)}{3n}\enspace.
\]
Using Cauchy-Schwarz inequality and assumption \eqref{Hyp.Sel.Est.Class}, we deduce that, if $m_{o}$ is a minimizer of $\norm{s_{m}-\st}^{2}+P\Psi_{m}/n$,
\begin{align*}
\var_{P}(s_{m}-s_{m_{o}})&\leq \norm{s_{m}-s_{m^{\prime}}}_{\infty}P|s_{m}-s_{m_{o}}|\\
&\leq \eta_{\infty,2}P(\Psi_{m}+\Psi_{m_{o}})\norm{s}\norm{s_{m}-s_{m_{o}}}^{2}\\
&\leq 8n\eta_{\infty,2}\norm{s}\paren{\norm{s_{m}-\st}^{2}+\frac{P\Psi_{m}}n}^{2}\enspace .
\end{align*}
Actually, we have $\norm{\st-s_{m_{o}}}^{2}\leq \norm{\st-s_{m_{o}}}^{2}+P\psi_{m_{o}}/n\leq \norm{s_{m}-\st}^{2}+P\Psi_{m}/n$, hence, by the triangular inequality,
\begin{align*}
\norm{s_{m}-s_{m_{o}}}^{2}\leq 2\norm{s_{m}-\st}^{2}+2\norm{\st-s_{m_{o}}}^{2}\leq 4\paren{\norm{s_{m}-\st}^{2}+\frac{P\Psi_{m}}n}\enspace .
\end{align*}
Moreover, $P\Psi_{m_{o}}\leq n\norm{\st-s_{m_{o}}}^{2}+P\psi_{m_{o}}\leq n\norm{s_{m}-\st}^{2}+P\Psi_{m}$, hence,
\[
P(\Psi_{m}+\Psi_{m_{o}})\leq 2n\paren{\norm{s_{m}-\st}^{2}+\frac{P\Psi_{m}}n}.
\]
We also have
\begin{align*}
&\frac{\norm{s_{m}-s_{m_{o}}}_{\infty}\ln(4/(\delta\pi(m)))}{3n}\\
&\leq\sqrt{\frac{2\eta_{\infty,2}P(\Psi_{m}+\Psi_{m_{o}})\norm{s_{m}-s_{m_{o}}}^{2}\ln(4/(\delta\pi(m)))}{n}}\\
&\qquad \times\sqrt{\frac{\eta_{\infty,2}P(\Psi_{m}+\Psi_{m_{o}})\ln(4/(\delta\pi(m)))}{18 n}}\\
&\leq \frac{2\sqrt{2}}3\eta_{\infty,2}\ln\paren{\frac{4}{\delta\pi(m)}}\paren{\norm{s_{m}-\st}^{2}+\frac{P\Psi_{m}}n}^{3/2}\enspace .
\end{align*}
Using a union bound, we deduce that, with probability larger than $1-\delta/2$, $\forall \mM$,
\[
\absj{(P_{n}-P)(s_{m_{o}}-s_{m})}\leq\frac{8\sqrt{2}}3\varepsilon_{m}(\delta)\paren{\norm{s_{m}-\st}^{2}+\frac{P\Psi_{m}}n}\enspace.
\]
Let us then denote by $L_{1}(\epsilon)=3\epsilon^{-1}L_{0}+16\sqrt{2}/3$ and by 
\end{proof}

\subsection{Control of the infimum}
\begin{lemma}\label{lem:WithQuantile}
Let $X_{1:n}$ be i.i.d random variables and let $\delta>0$ such that $\PESup{\ln(\delta^{-2})}\leq n/2$. Let $\B$ be a regular partition of $\set{1,\ldots,n}$, with $V=\PESup{\ln(\delta^{-2})}\vee 8$.
Let $(\ERM_{K})$ be a sequence of estimators such that, for all $K$, $\ERM_{K}=f(X_{B_{K}})$. Let $\alpha\in (0,1)$ and let $x_{\alpha}$ be a real number satisfying the following property.
\[
\card\set{K=1,\ldots,V,\;\P\paren{\ell(\ERM_{K},s_{o})>x_{\alpha}}\leq e^{-\frac{1}{2\alpha}}}\geq \alpha V.
\]
Then we have
\[
\P\paren{\inf_{K=1,\ldots,V}\ell(\ERM_{K},s_{o})>x_{\alpha}}\leq \delta.
\]
\end{lemma}
\begin{proof}
By independence of the $X_{B_{K}}$, 
\[
\P\paren{\inf_{K=1,\ldots,V}\ell(\ERM_{K},s_{o})>x_{\alpha}}=\prod_{K=1}^{V}\P\paren{\ell(\ERM_{K},s_{o})>x_{\alpha}}\leq (e^{-\frac{1}{2\alpha}})^{\alpha V}\leq \delta.
\]
\end{proof}

\noindent
An elementary application of the previous result is the following lemma.
\begin{lemma}\label{lem:withmoments}
Let $X_{1:n}$ be i.i.d random variables and let $\delta>0$ such that $\PESup{\ln(\delta^{-2})}\leq n/2$. Let $\B$ be a regular partition of $\set{1,\ldots,n}$, with $V=\PESup{\ln(\delta^{-2})}\vee 8$.
Let $(\ERM_{K})$ be a sequence of estimators such that, for all $K$, $\ERM_{K}=f(X_{B_{K}})$. Let $\alpha\in (0,1)$ and let $E_{\alpha}\geq0$ such that
\[
\card\set{K=1,\ldots,V,\;\E(\ell(\ERM_{K},s_{o}))\leq E_{\alpha}}\geq \alpha V.
\]
Then, we have
\[
\P\paren{\inf_{K}\ell(\ERM_{K},s_{o})\leq e^{\frac{1}{2\alpha}}E_{\alpha}}\geq 1-\delta.
\]
\end{lemma}
\begin{proof}
From Markov inequality, 
\[
\P\paren{\ell(\ERM_{K},s_{o})\leq e^{\frac{1}{2\alpha}}\E\paren{\ell(\ERM_{K},s_{o})}}\leq e^{-\frac{1}{2\alpha}}.
\]
Hence, the result follows from Lemma \ref{lem:WithQuantile} and the assumption on $E_{\alpha}$.
\end{proof}
\begin{lemma}\label{lem.for.Regression}
Let $X_{1:n}$ be i.i.d random variables with common measure $P$. Let $\B$ be a regular partition of $\set{1,\ldots,n}$ with cardinality larger than $\PESup{\ln(\delta^{-2})}$. Let $r=(1-\sqrt{1-e^{-2}})/2$ and assume that there exist a collection of independent events $(\Omega_{K})_{K=1,\ldots,V}$ such that
\begin{align*}
\forall K=1,\ldots,V,\qquad \E\paren{\ell(\ERM_{K},s_{o})\un_{\Omega_{K}}}\leq E,\qquad \P\set{\Omega_{K}^{c}}\leq r\enspace .
\end{align*}
Then, we have
\[
\P\paren{\inf_{K=1,\ldots,V}\ell(\ERM_{K},s_{o})>4e^{2}E}\leq 2e^{-V}\enspace .
\]
\end{lemma}
\begin{proof}
We introduce the random variables $Y_{K}=\un_{\Omega_{K}^{c}}$. We apply Lemma \ref{lem.coupl} with $x=1/2$, and $p=r$. Denoting by $B(V,r)$ a random variable with binomial distribution of parameters $V$ and $r$, we obtain
\[
\P\set{\card\set{K=1,\ldots,V,\telque \Omega_{K}^{c}\;\mbox{holds}}\geq \frac{V}2}\leq \P\set{B(V,r)\geq\frac V2}\enspace .
\]
Let us denote by $\Omega=\set{K=1,\ldots,V,\telque \Omega_{K}^{c}\;\mbox{holds}}$ and $N=\card\Omega$. By inequality (2.8) in \cite{Ma07}
\[
\P\set{N\geq \frac{V}2}\leq e^{-\frac V2\ln\paren{\frac1{4r(1-r)}}}=e^{-V}\enspace .
\]
As a consequence,
\begin{equation}
\P\set{V-N<\frac{V}2}\leq e^{-V}\enspace .
\end{equation}
Moreover, for all $x>4$, we have
\begin{align*}
\P&\set{\inf_{K=1,\ldots,V}\ell(\ERM_{K},s_{o})>xE\cap N<\frac{V}2}\\
&=\sum_{A\subset\set{1,\ldots,n},\;\card{A}\geq \frac V2}\P\set{\inf_{K=1,\ldots,V}\ell(\ERM_{K},s_{o})>xE\cap \Omega^{c}=A}\\
&\leq\sum_{A\subset\set{1,\ldots,n},\;\card{A}\geq \frac V2}\P\set{\inf_{K\in A}\ell(\ERM_{K},s_{o})>xE\cap \Omega^{c}=A}\\
&\leq\sum_{A\subset\set{1,\ldots,n},\;\card{A}\geq \frac V2}\P\set{\inf_{K\in A}\ell(\ERM_{K},s_{o})\un_{\Omega_{K}}>xE}\\
&\leq \sum_{A\subset\set{1,\ldots,n},\;\card{A}\geq \frac V2}\paren{\frac1x}^{\card{A}}\enspace .
\end{align*}
We conclude the proof with the fact that, if $x=4e^{2}$, $(1/x)^{\card{A}}\leq x^{-V/2}\leq (2e)^{-V}$ and the fact that there is less than $2^{V}$ terms in the sum.
\end{proof}

\bibliographystyle{plain}
\bibliography{bibliolerasle}
\end{document}